\documentclass{article} 
 
\usepackage{graphicx}
\usepackage{hyperref}
\usepackage{stmaryrd}
\usepackage{graphicx}
\usepackage{amsmath}
\usepackage{mathrsfs} 
\usepackage{amsfonts}
\usepackage{amssymb}%
\usepackage{amsthm}
\usepackage{amscd}
\usepackage{url}
\usepackage[all]{xypic} 
\usepackage{endnotes}
\usepackage{tikz}

\newcommand{\mcM}{\mathcal{M}}

\newcommand{\scrG}{\mathscr{G}}
\newcommand{\scrJ}{\mathscr{J}}
\newcommand{\scrP}{\mathscr{P}}
\newcommand{\PP}{\mathbb{P}}

\newcommand{\C}{\mathbb{C}}

\newcommand{\N}{\mathbb{N}}

\newcommand{\FF}{\mathbb{F}}

\newcommand{\x}{\mathbf{x}}

\newcommand{\isom}{\cong}

\newcommand{\ot}{\otimes}
\newcommand{\by}{\! \times \!}

\newcommand{\Sym}{\operatorname{Sym}}
\newcommand{\Spin}{\operatorname{Spin}}
\newcommand{\val}{\operatorname{val}}

\newcommand{\GL}{\operatorname{GL}}
\newcommand{\PGL}{\operatorname{PGL}}
\newcommand{\sign}{\operatorname{sign}}

\newcommand{\ceil}[1]{\ensuremath{\lceil #1 \rceil}}

\newcommand{\sPf}{\operatorname{sPf}}
\newcommand{\sPfd}{\operatorname{sPf}^{\vee}}
\newcommand{\Pf}{\operatorname{Pf}}

\newcommand{\union}{\cup}
\newcommand{\Pol}{\operatorname{Pol}}
\newcommand{\Inv}{\operatorname{Inv}}

\newcommand{\bra}{\langle} 
\newcommand{\ket}{\rangle} 
\newcommand{\la}{\langle}

\newcommand{\grid}[3]{
\begin{scope}[yshift=5cm]
\foreach \row in {1, 4,..., #1}
{
	\foreach \x in {1.73205, 3.4641, ..., #2}
	{
	\begin{scope}[] 
		\draw[gray,very thin] (\x, \row) +(30:1cm) -- +(90:1cm) -- +(150:1cm) -- +(210:1cm) -- +(270:1cm) -- +(330:1cm) -- cycle;
	\end{scope}
	}
	\foreach \x in {1.73205, 3.4641, ..., #3}
	{
	\begin{scope}[xshift = 0.8660254 cm, yshift = -1.5 cm]
		\draw[gray,very thin] (\x, \row) +(30:1cm) -- +(90:1cm) -- +(150:1cm) -- +(210:1cm) -- +(270:1cm) -- +(330:1cm) -- cycle;
	\end{scope}
	}
}
	\foreach \x in {1.73205, 3.4641, ..., #2}
	{
	\begin{scope}[yshift = -2cm]
		\draw[gray,very thin] (\x, 0) +(30:1cm) -- +(90:1cm) -- +(150:1cm) -- +(210:1cm) -- +(270:1cm) -- +(330:1cm) -- cycle;
	\end{scope}
	}
\end{scope}
}

\theoremstyle{plain}
\newtheorem{theorem}{Theorem}[section]
\newtheorem{corollary}[theorem]{Corollary}
\newtheorem{proposition}[theorem]{Proposition}
\newtheorem{lemma}[theorem]{Lemma}

\newtheorem{observation}[theorem]{Observation}

\newtheorem{problem}[theorem]{Problem}
\theoremstyle{definition}
\newtheorem{defn}[theorem]{Definition}
\theoremstyle{definition}

\theoremstyle{definition}
\newtheorem{example}[theorem]{Example}
\theoremstyle{definition}
\newtheorem{algorithm}{Algorithm}

\title{Pfaffian circuits}

\author{Jason  Morton\thanks{Supported by the Defense Advanced Research Projects Agency under Award No. N66001-10-1-4040.}}

\begin{document}
\maketitle

\begin{abstract} 
It remains  an open question whether the apparent additional power of quantum computation derives inherently from quantum mechanics, or merely from the flexibility obtained by  ``lifting'' Boolean functions to linear operators and evaluating their composition cleverly.  Holographic algorithms provide a useful avenue for exploring this question.  We describe a new, simplified construction of holographic algorithms in terms of Pfaffian circuits.  Novel proofs of some key results are provided, and we extend the approach of \cite{HAWMG} to nonsymmetric, odd, and homogenized signatures, circuits, and various models of execution flow.  This shows our approach is as powerful as the matchgate approach.    
Holographic algorithms provide in general $O(n^{\omega_p})$ time algorithms, where $\omega_p$ is the order of Pfaffian evaluation in the ring of interest (with $1.19 \leq \omega_p \leq 3$ depending on the ring) and $n$ is the number of inclusions of variables into clauses.  Our approach often requires just the evaluation of an $n \times n$ Pfaffian, and at most needs an additional two rows per gate, whereas the matchgate approach is quartic in the arity of the largest gate.  
We give examples (even before any change of basis) including efficient algorithms for certain lattice path problems and an $O(n^{\omega_p})$ algorithm for evaluation of Tutte polynomials of lattice path matroids.
Finally we comment on some of the geometric considerations in analyzing Pfaffian circuits under arbitrary basis change.  Connections are made to the sum-product algorithm, classical simulation of quantum computation, and SLOCC equivalent entangled states.
\end{abstract}

\section{Introduction}

Valiant's \cite{QCtcbSiPT} matchgates give a variety of linear operators on two qubits which can be composed according to certain rules into matchcircuits, thereby simulating certain quantum computations in polynomial time.  An alternative formulation into {\em planar matchgrids} is also available \cite{ValiantFOCS2004}, and is equivalent \cite{MR2305569}.  
In this paper we detail a third formulation.  
If a circuit is built from matchgates, and the input and any bit of the output is fixed, the probability of observing that bit can be computed in classical polynomial time.  Valiant's two-input, two-output matchgates are characterized by two linear operators acting independently on the even and odd parity subspaces of a pair of qubits, with the restriction that the operators have the same determinant.  More generally matchgates are characterized by the vanishing of polynomials called {\em matchgate identities}.  These gates correspond to linear optics acting on non-interacting fermions with pairwise nearest neighbor gates and one-bit gates \cite{TerhalDiVincenzo2002,JozsaMiyake}.
  The theory of matchgates has been further developed by Cai et al., who explained the relation to tensor contractions \cite{MR2277247}.

We describe a framework for constructing and analyzing holographic algorithms which is equivalent to the matchgate formulation, but which does not use matchgates.  This builds upon our results in \cite{HAWMG}. 
The main contributions of the present work are as follows.  Theorem \ref{thm:otimesoplus} gives a new, explicit Pfaffian ordering together with a proof of its validity. Section \ref{ssec:homogodd} extends our construction \cite{HAWMG} to asymmetric, odd, and homogenized signatures, which shows that it is as powerful as the matchgate approach.  We then analyze the cost of this extension.  Given a $\mathsf{\#CSP}$ instance with $n$ inclusions of variables into clauses, our construction often requires just the evaluation of the Pfaffian of an $n\times n$ matrix.  Theorem \ref{thm:twonewedges} states that in the worst case, the extension requires an additional two matrix rows per homogenized predicate.  Section \ref{ssec:stdbasisapplications} derives some new algorithms for lattice path enumeration using these results.  Finally, Theorem \ref{thm:freebasisarity3} gives the first result on holographic algorithms with heterogeneous basis change: all arity-three predicates are implementable.

In Section \ref{sec:prelim}, we give necessary background, definitions, and notation.  In Section \ref{sec:tensorcontractioncircuits}, we describe the tensor contraction circuit formalism for readers unfamiliar with it.  In Section \ref{sec:HoloCircuits}, we introduce Pfaffian circuits and the method for evaluating them in polynomial time. 
In Section \ref{sec:CBholographic}, we analyze simple and compound Pfaffian predicates in the standard basis, which all holographic computations are reduced to before evaluation, and relate them to spinor varieties.  We also derive some algorithms for lattice path enumeration problems as examples. 
Finally in Section \ref{sec:basischange} we comment on some of the geometric considerations in analyzing Pfaffian circuits under arbitrary basis change.

\section{Preliminaries}\label{sec:prelim}

It is by now well known \cite{Beaudry07, Bravyi2008, MR2277247, Damm03, Markov2008, Montanari, NielsenChuang, QCtcbSiPT} that the counting version of a Boolean constraint satisfiability problem ($\mathsf{\#CSP}$) can be expressed as a tensor contraction.  Moreover, such tensor contraction {\em networks} or {\em circuits} serve as good general models for computation capable of being specialized to classical Boolean circuit computation, quantum computing, and other settings.  Restricting the set of tensor circuits allowed in various ways, one can accurately capture several complexity classes including Boolean circuits, $\mathsf{BQP}$, $\mathsf{\#P}$, etc. \cite{Damm03}.   

It is an open question whether the additional power of the quantum computational model\footnote{An example of how a quantum computation can be viewed as a $\mathsf{\#CSP}$ tensor contraction is the following.  All the preparations and operations are naturally expressed in terms of tensor products and contractions; indeed a quantum circuit diagram is a type of tensor contraction network.  Suppose $n$ input bits are placed in uniform superposition.  A unitary operator is applied which represents a CSP and outputs many garbage bits and one satisfiable-or-not bit $s$.  Then the probability of observing $s\!=\!1$ is the number of satisfying assignments divided by $2^n$.}
 comes inherently from the special features of quantum mechanics, or merely from the additional freedom which comes from ``lifting'' Boolean functions by viewing them instead as linear operators.  After such a lifting, we may study the complete computation as a tensor contraction, which may have a surprisingly efficient algorithm to evaluate.  Indeed, this has been a known open avenue to a possible (full or even partial \cite{aaronson2010computational,bremner2010classical}) collapse of the polynomial hierarchy. 

\subsection{Tensors and predicates}
Let $V^1, V^2, \dots, V^n$ be $2$-dimensional vector spaces and $V^{1\vee}, V^{2\vee}, \dots, V^{n\vee}$ their duals.  Fix a basis $\{v^i_0, v^i_1\}$ for each $V^i$; this will be called the standard basis (e.g. in quantum computing, these could represent choices of orthogonal bases).  Let $\{\nu^i_0, \nu^i_1\}$ be the dual basis of $V^{i*}$, so $\nu^i_j (v^i_k)=\delta_{jk}$.  Fixing an order of the $V^i$ we denote an induced basis element of $V^1 \ot V^2 \ot \cdots \ot V^n$ by a bitstring such as 
\[| 0100 \cdots 10 \ket = v^1_0 \ot v^2_1 \ot v^3_0 \ot v^4_0 \ot  \cdots \ot v^{n-1}_1 \ot v^n_0\]
and e.g. $\langle 00111 \cdots 10|$  (n bits) for an element of the induced dual basis.  Note that we omit the scripts identifying the vector spaces when the correspondence intended is clear from context. 
   When we need to keep track of the vector spaces involved, we write 
\begin{equation}\label{eq:labelledbraket}
| 0_1 1_4 0_7 \ket \quad \text{for} \quad v^1_0 \ot v^4_1 \ot v^7_0.
\end{equation}
A bitstring $x$ is sometimes abbreviated by the 
set of ones in the bitstring; e.g.\ we write  
$|J\ket$, with $J=\{3,4\}$ for the induced basis element $| 0011 \ket$ when $n=4$.  

A Boolean predicate or relation is a formula in a Boolean algebra (e.g. $p\!\implies\!q, \; p \vee q$), i.e. a truth table.  By designating some variables as inputs and others outputs, such that all possible inputs have a unique output, a predicate can also be viewed as a function or gate.  If there are more than one output, it could be viewed as a nondeterministic gate; with an input missing, a gate capable of terminating some computation paths.  

With the bitstring bra-ket notation, it is natural to express Boolean predicates and functions in the same terms as arbitrary linear transformations or tensors.  A predicate becomes the formal sum of the rows of its truth table as bitstrings.  For example $OR_3 = (| 0\ket+| 1 \ket)^{\ot 3}-| 000\ket$. Thus any Boolean predicate can be viewed as a multilinear operator.  Thus any Boolean predicate is a multilinear operator.  When we want to think of, e.g.  $AND_{2,1}=| 00,0\ket+| 01,0\ket+| 10,0\ket+| 11,1\ket$ as a gate or function accepting two bits as input and outputting another, we separate the input and output by a comma.  However nothing has really changed and this is just to aid in reasoning about a circuit when we orient its edges (see Figure \ref{fig:XORs}).  Viewed as linear operators in an explicit basis, predicates may be composed by matrix multiplication.  Tensors with a mix of primal and dual vector spaces have mixed {\em ality}.  Say $| 0_10_21_3 \ket$ is a ${ 0 \choose 3}$ ality tensor and $\bra 1_41_51_61_7 |$ is a ${4 \choose 0}$ tensor, $| 0_10_21_3 \ket \la 1_41_51_61_7 |$ is ${ 4 \choose 3}$.  If instead we have  $\bra 0_21_31_41_5 |$, the partially contracted $ \la 0_21_31_41_5 | 0_10_21_3 \ket$ is a ${2 \choose 1}$ tensor equal to $|0_1\ket\bra 1^41^5|$.  A ${0 \choose 0 }$ tensor (e.g.  $\la 0_10_21_3 | 0_10_21_3 \ket=1$) is a scalar.
A {\em gate} or  {\em $n$-gate} is a ${0 \choose n}$ tensor of this type (e.g. $OR_3$, $AND_{2,1}$), and a {\em cogate} or  {\em n-cogate} is a  ${n \choose 0}$ tensor.  A general ${n \choose m}$ tensor is a {\em predicate}.  The {\em arity} of an ${n \choose m}$ predicate is $n+m$, the tensor degree or number of edges incident when the predicate is included in a circuit.

We can define a restricted class $\scrP$ of allowed predicates.  Formally these are a (possibly infinite) collection of unlabeled predicates $|I\ket$ which map subsets of $[n]$ of size $|I|$ to labeled predicates, in a computation of size $n$.   These generate a space of possible circuits or tensor formulae called compound predicates by tensor product and partial contraction.  A generating set of predicates which enables us to construct all Boolean functions is said to be universal.  Because swap and fanout are not always included in such a generating set, we cannot simply appeal to results such as Post's lattice \cite{post_two-valued_1941} to analyze the space of circuits computable by a given generating set of predicates.  Worse, for holographic algorithms there are further restrictions on how available predicates can be composed to form circuits.  In some cases, merely determining whether one predicate is implementable in terms of another can be undecidable \cite{CookPredicates}.
\begin{theorem}[\cite{CookPredicates}]
Without fanout, there is no general procedure for deciding the following question: Given a predicate $X$, can a graph of $X$'s be built that implements a given desired predicate $Y$?
\end{theorem}
The study of holographic algorithms and related questions is in large part concerned with studying certain restricted families of predicates, subject to additional composition rules, and asking what the resulting space of circuits can compute.  The result is an interplay between the algebraic geometry which describes the space of allowed predicate families and the combinatorial and computational complexity theory needed to analyze the power of compositions to compute useful things.

In the language of universal algebra \cite{bulatov2005classifying}, which has been central to the recent $\mathsf{\#CSP}$ dichotomy theorem \cite{bulatov2010complexity} over finite relational structures (i.e.\, for combinatorially unrestricted circuits of unlifted predicates), our construction can be described as follows.  Let $A$ be a finite alphabet and $A^n$ the set of all $n$-tuples of elements of $A$ or {\em n-ary relations over $A$}.  Let $R_A = \union_{n \in \N} A^n$ be the set of all finitary relations over $A$.  A constraint language, or restricted class of predicates $L$ is subset of $R_A$.  The set of implementable predicates is characterized by a closure with respect to invariants.
\begin{theorem}[\cite{bodnarchuk1969galois,geiger1968closed}]
Without variable or combinatorial restrictions (i.e.\ with the presence of fanout and swap), the set of predicates implementable by compositions of a given set of predicates $L$ is the {\em clone} of $L$, $\Inv(\Pol(L))$.  
\end{theorem}
This yields a Galois correspondence between $R_A$ and sets of finitary operations on $A$ analogous, for example, to the correspondence between ideals and algebraic varieties.  See \cite{bohler2003playing, bohler2004playing, bulatov2005classifying} for details.
Denote by $\C R_A$ the tensor algebra built from the vector space $\C A$ spanned by the elements of $A$.  
This is just the noncommutative ring over $\C$, whose monoid has elements $n$-tuples of elements of $A$ and semigroup operation tuple concatenation. 
There is a natural embedding 
which we call {\em lifting} of $R_A$ into $\C R_A$ which sends an $n$-tuple of $A$ to the corresponding induced basis vector.  We have the additional wrinkle of putting variables and constraints on the same footing, so that we need both a constraint language $L \subset R_A$ and a variable language  $V \subset R_A$ (in fact lifted versions of both).  In this paper we restrict attention to the binary case $|A|=2$.  We also note that this lifting can be profitably expressed in terms of functors of monoidal categories, but postpone details of this approach to future work.

Given such a restricted class of predicates, introducing a change of basis expands, or more properly rephrases, the set of predicates available.  To the extent the rephrasing allows us to work with familiar Boolean predicates or close relatives, this makes the computational complexity aspect more manageable.  However the two predicates connected by a shared vector space, one with the primal and one with the dual, must use the same basis on that vector space.  Suppose $A$ is the change of basis on such a vector space $V$, with
\[
A= \bordermatrix{
&| 0 \ket & | 1 \ket \cr
| 0 \ket & a_{00} & a_{10}\cr
| 1 \ket& a_{01} & a_{11}\cr
}
\]
so that $A:| 0 \ket \mapsto  a_{00}| 0\ket+a_{01}| 1 \ket$ and  $ | 1 \ket \mapsto  a_{10} | 0\ket+a_{11}| 1 \ket$.  Then the dual change of basis applying $(A^{-1})^{\top}$ on $V^{\vee}$ is $A^{\vee}:\la 0 | \mapsto (\det A)^{-1} \big ( a_{11}\la 0 |  - a_{10} \la 1 | \big )$ and  $ \la 1 | \mapsto (\det A)^{-1}\big ( -a_{01} \la 0|+a_{00}\la 1| \big )$.  We can see that $1= \la 0 | 0 \ket = \la A^{\vee}(\la 0 |), A( | 0 \ket ) \ket = (\det A)^{-1} \big ( a_{11}a_{00} \la 0 | 0\ket -  a_{10}a_{01} \la 1 | 1\ket \big ) = 1$ 
and similarly for $\la 1 | 1 \ket$; moreover $\bra 0 | 1 \ket = \bra 0 | 1 \ket=0$.  We record this observation as follows.
\begin{proposition} \label{prop:basischangepairing}
Applying a change of basis $A$ to any vector space factor of the tensor space of a contraction and $A^{\vee}$ to its dual does not affect the value of the pairing. 
\end{proposition}
However, applying a change of basis on each vector space of a pairing may enable us to simplify its evaluation.  The process is somewhat analogous to simplifying or decomposing a matrix or tensor to a more manageable form by Givens rotations.  Here the ``rotations'' serve to translate between predicates which are convenient to reason about in computational complexity terms (e.g. Boolean predicates) and predicates whose composition can be evaluated efficiently.

A very useful change of basis is the Hadamard change of basis
\[
H= \begin{pmatrix}
 1 & 1\cr
 1 & -1\cr
\end{pmatrix}
\]
mapping  $| 0 \ket \mapsto  | 0\ket+| 1 \ket$ and  $ | 1 \ket \mapsto  | 0\ket - | 1 \ket$; we have $H^{\vee}=\frac{1}{2}H$. We avoid the more natural self-dual $\sqrt{2}H$ only to maximize the opportunity for exact integer arithmetic.   In most cases any constants can be ignored until organizing the final computation.

\begin{example} The Not-All-Equal clause is true if the variables are not all equal
\[
NAE_3 = |001\ket + |010\ket + |100\ket + |011\ket + |101\ket + |110\ket.
\]
Under the Hadamard change of basis it equals
\[
H^{\ot 3} NAE_3 = 6|000\ket -2 |011\ket -2 |101\ket -2 |110\ket ,
\]
which we will see makes it Pfaffian, and hence tractable to compute with under planarity.  Note that the basis-changed version only makes sense in $\C R_A$  and not $R_A$.
\end{example}

\section{Tensor contraction circuits} \label{sec:tensorcontractioncircuits}
A circuit is a composition of predicates that computes something of interest. Given a collection of (co)gates (or predicates), we may diagram their contraction by means of a bipartite graph which shows which pairing to make.  
\begin{defn} \label{defn:tensconcirc}
A {\em tensor contraction circuit} (or simply {\em circuit}) $\Gamma$ is a 
combinatorial object consisting of 
\begin{itemize}
\item[(i)] a set $\scrP$ of tensor predicates with coefficients in a field $\FF$, the gadgets we may use to build the circuit, and 
\item[(ii)] a graph $\Gamma = (\text{Predicates}, \text{Edges})$ such that each predicate in Predicates is drawn from $\scrP$, each edge of a predicate corresponds to one of its vector spaces, and two predicates sharing an edge have dual vector spaces corresponding to that edge.
\end{itemize}
\noindent This $\Gamma$ represents the circuit architecture: which vector spaces are paired, including the assignment for non-symmetric predicates.
\end{defn}
The {\em value} $\val(\Gamma)$  of a circuit is the result of the indicated maximal tensor contraction.  A circuit is {\em closed} if the contraction is total: all vector spaces appearing are paired and the result is a field element (Figure \ref{fig:XORs}(b),(c)).  This field element is the ``answer'' to the corresponding $\#\mathsf{CSP}$ instance the circuit represents, computed over the field (more generally, we could use a ring or even semiring).
In the statistical setting, it is the partition function.  If there are unmatched wires or ``dangling edges,'' the circuit is {\em open} and its value is a tensor (Figure \ref{fig:XORs}(a)).  
In the statistical setting, this is a marginal distribution over a subset of variables, or the most likely assignment if working over the tropical semiring \cite{Pachter16112004}. 
Predicates in $\scrP$ (below, predicates of the form $\sPf \Xi$ or $\sPfd \Theta$), are called {\em simple}.  A {\em compound} gate or cogate is a partial circuit composed of predicates in $\scrP$, which represents a partial pairing but still has dangling edges, and for which the ality is that of a (co)gate respectively after the partial contraction is performed (Figure \ref{fig:XORSWAP}).  A general compound predicate can have mixed ality.  
\begin{observation}
In reasoning about the value of a closed circuit, we may treat compound predicates as black boxes: they will behave exactly as though there were an actual predicate in terms of their contribution to the circuit value.
\end{observation}

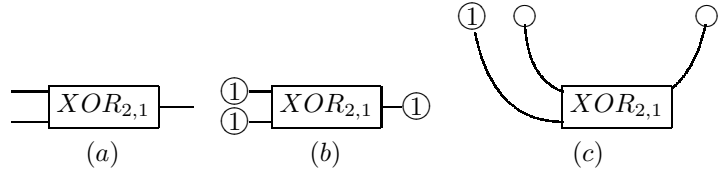
\begin{figure}
\[
\begin{array}{ccc}
\begin{xy}<10mm,0mm>:
(0,0) *+{XOR_{2,1}}*\frm{-};
  p+(1.2,0) **@{-},
(-.72,.2);   p+(-.5,0)  **@{-},
(-.72,-.2);   p+(-.5,0)  **@{-},
\end{xy}
& 
\begin{xy}<10mm,0mm>:
(0,0) *+{XOR_{2,1}}*\frm{-};
  p+(1.2,0)  *+{1}*\frm{o} **@{-},
(-.72,.2);   p+(-.5,0) *+{1}*\frm{o} **@{-},
(-.72,-.2);   p+(-.5,0)  *+{1}*\frm{o}  **@{-},
\end{xy}
&
\begin{xy}<10mm,0mm>:
(0,0) *+{XOR_{2,1}}*\frm{-};
  p+(1.2,1.2)  *+{\;}*\frm{o} **\crv{~*=\dir{.}p+(1,0)},
(-.72,.2);   p+(-.5,1) *+{\;}*\frm{o}  **\crv{~*=\dir{.}p+(-.5,.2)},
(-.72,-.2);   p+(-1.2,1.4)  *+{1}*\frm{o}  **\crv{~*=\dir{.}p+(-1,0)},
\end{xy}\cr 
(a)&(b)&(c)\cr
\end{array}
\]
\caption{Viewing $XOR_{2,1}$ as (a) an open circuit, logic gate, or function; on the left it ``accepts inputs'' and on the right ``outputs'' their sum modulo two.  After pairing all vector spaces, we have (b) a closed circuit whose value is 0 because the indicated assignment is not in the graph of the function.  In (c) we show the $\mathsf{\#CSP}$ interpretation; the value of this closed circuit is 2 because two of the four possible assignments to the open circles (representing summation), $\bra 0|\bra 1|$  and $\bra 1|\bra 0|$, yield elements of the function's graph (contract to 1).  That is, $\val(\Gamma)=(\bra 1|)( \bra 0| + \bra 1|) ( \bra 0| + \bra 1|) \cdot (XOR_{2,1}) = (\bra 100 + \bra 101| + \bra 110| + \bra 111|) \cdot (XOR_{2,1}) = 2$. } \label{fig:XORs}
\end{figure}
\begin{figure}
\[
\begin{array}{ccc}
\begin{xy}<10mm,0mm>:
(0,0);
  p+(0,-1) *+{XOR}*\frm{-} **@{-} ?>*\dir{>};
  p+(0,-.45);p+(0,-.45)  **@{-} ?>*\dir{>};
  p+(0,-.3)  *+{=_3}*\frm{o};
  p+(1.4,0)  **@{-} ?>*\dir{>},
  {p+(0,-.3);p+(0,-.8) *+{XOR}*\frm{-} **@{-} ?>*\dir{>}};
  p+(0,-.45);  p+(0,-.8) **@{-} ?>*\dir{>};
(2,0);
  p+(0,-.7)  **@{-} ?>*\dir{>};
  p+(0,-0.3) *+{=_3}*\frm{o};
  p+(-1.4,0)  **@{-} ?>*\dir{>},
  p+(0,-.25);p+(0,-.5)  **@{-} ?>*\dir{>};
  p+(0,-0.25) *+{XOR}*\frm{-};
  p+(0,-.8)  **@{-} ?>*\dir{>};
  p+(0,-0.3) *+{=_3}*\frm{o};
  p+(-1.4,0)  **@{-} ?>*\dir{>},
  p+(0,-0.3);p+(0,-.7)  **@{-} ?>*\dir{>},
\end{xy}
&\quad &
\begin{xy}<10mm,0mm>:
(0,0);
  p+(0,-1) *+{XOR}*\frm{-} **@{-};
  p+(0,-1.3)  *+{=_3}*\frm{o} **@{-};
  p+(1.5,0)  **@{-},
  p+(0,-1) *+{XOR}*\frm{-} **@{-};
  p+(0,-.8) **@{-};
  p+(-1.5,4) **\crv{~*=\dir{.}p+(-.5,-1)},
(2,0);
  p+(0,-1)  *+{=_3}*\frm{o}  **@{-} ;
  p+(-1.5,0)  **@{-},
  p+(0,-1.3)    *+{XOR}*\frm{-} **@{-};
  p+(0,-1.1)  *+{=_3}*\frm{o}  **@{-};
  p+(-1.5,0)  **@{-},
  p+(0,-.7)  **@{-};
  p+(1.5,4) **\crv{~*=\dir{.}p+(.5,-1)},
\end{xy}\cr
(a)&\quad &(b)
\end{array}
\]
\caption{In (a), a swap gate built of $XOR_{2,1}$ (or $EVO_3$) and $FANOUT_{1,2}$ (also called $=_3$).  Multiple  orientations of the edges produce valid functional interpretations on the circuit.  Reversible gates are a special case of the freedom to orient in any direction that makes each predicate act as a function.  In (b), arrows are removed and edges curved to the top to suggest a $\mathsf{\#CSP}$ interpretation.} \label{fig:XORSWAP}
\end{figure}
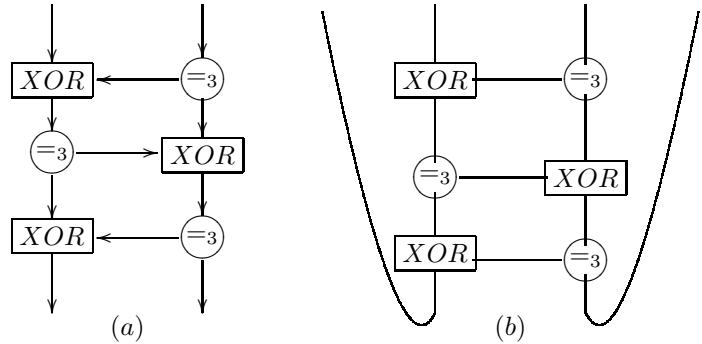

For example, to represent a \#3SAT problem, we could use an arity-$d$ all-equal ($AE^d$) cogate to represent each variable appearing in $d$ clauses, $OR_3$ gates to represent clauses, with an edge if a variable is included in a clause.  The value of such a circuit is the number of satisfying assignments.  If the tensor coefficients are not in the set $\{0,1\}$, it is called a weighted sum or weighted $\mathsf{\#CSP}$ problem and its value is a weighted sum or polynomial.  Note that even in the weighted setting, an open circuit may still have $\{0,1\}$ coefficients after the partial contraction is performed; indeed it can be worthwhile to engineer things this way to implement Boolean functions with  restricted sets of (possibly non-Boolean) predicates.

Open circuits are useful conceptually, and ``evaluate'' to tensors, i.e. multilinear transformations.  However, we primarily evaluate only closed Pfaffian circuits $\Gamma$, as these are the type which have a general polynomial time algorithm (Algorithm \ref{alg:hologeval}) to obtain $\val(\Gamma)$.   
An open circuit is a bit like a quantum computation before, and a closed circuit a quantum computation after, inputs and measurements have been made.

If the dangling edges of an open circuit can be partitioned into a set of $k$ inputs and $\ell$ outputs, such that every possible input in $\{0,1\}^k$ has a unique nonzero coefficient for some assignment of values in $\{0,1\}^{\ell}$ to the outputs, we can view the open circuit as representing a function; the open circuit's value is a tensor which is the formal sum of points in the graph of the function, $\{\sum |x\ket \ot f(|x\ket):x \in \{0,1\}^k \}$.  If the output assignment is not unique, one might interpret this as a nondeterministic function. This is useful in using an orientation to compute a circuit value as in Section \ref{sec:factorization}.  We can also consider orientations of the edges that mark inputs and outputs appropriately.  If we then make a partial assignment to the input and output wires, and sum over the rest, we can compute various things about the function so represented.  For example, suppose the circuit is a multiplication circuit, so the input is the bit representation of two integers and the output their product.  Fix the output to be a given integer.  Now trying a zero and one assignment to each bit of input and summing over the rest, and checking to see if the resulting closed circuit value is nonzero, we can  find a factors of the given integer by binary search.

The notion of circuit value  is inherently of the counting type, but can be specialized to answer decision, function, and \#function questions.  
Input, output, and \#function vs \#decision problems are just a matter of perspective on time flow through the circuit.  The fact that tensor contraction networks can be interpreted in more than one way depending on the orientation or flow of time chosen is familiar to physicists in the guise of the time-invariance of Feynman diagrams.

\subsection{Evaluating tensor circuits with factorization} \label{sec:factorization}
The sum-product algorithm \cite{Aji2000,Kschischang01factorgraphs}, which computes certain exponentially large sums of products in polynomial time,  is one of the most ubiquitous algorithms in applications.  It is a strategy for computing the value of a tensor circuit which has been rediscovered many times under different names. The sum-product algorithm is used in graphical models (belief propagation and junction tree), computer vision, hidden Markov models (forward-backward algorithm), genetics (Needleman-Wunsch), natural language processing, sensor networks \cite{Ihler2004}, turbo coding \cite{Kschischang01factorgraphs}, quantum chemistry, simulating quantum computing \cite{Markov2008}, SAT solving (survey propagation) and elsewhere.  
The sum-product algorithm applies to problems phrased in terms of requirements, clauses, or factors (here, ``gates'') on the one hand and variables (``cogates'') on the other, and for which the bipartite {\em factor graph} with an edge $(v,c)$ if variable $v$ appears in clause $c$ is a tree.  This factor graph is precisely a tensor contraction network or circuit as described previously.  It works over any semi-ring and exploits the tree structure to factor the problem, thus reducing a seemingly exponential problem to one which can be computed in polynomial time.  However, despite approximate generalizations \cite{Murphy99loopybelief}, the algorithm is fundamentally limited to the case where the underlying factor graph is a tree.

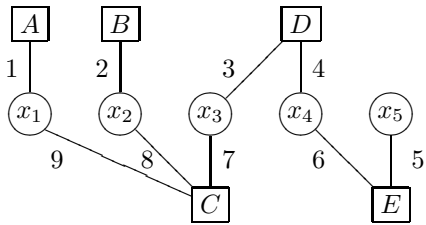
\begin{figure}[htb]
\[
\begin{xy}<12mm,0mm>:
(1,0) *+{x_1}*\frm{o};
p+(0,.25);p+(0,.75) *+{A}*\frm{-} **@{-},
(2,0) *+{x_2}*\frm{o};
p+(0,.25);p+(0,.75) *+{B}*\frm{-} **@{-},
(3,0) *+{x_3}*\frm{o};
p+(0,-.25);p+(0,-.75) *+{C}*\frm{-} **@{-};
p+(-.82,.82) **@{-},
p+(-1.82,.82) **@{-},
(4,0) *+{x_4}*\frm{o};
p+(0,.25);p+(0,.75) *+{D}*\frm{-} **@{-};
p+(-.82,-.82) **@{-},
(5,0) *+{x_5}*\frm{o};
p+(0,-.24);p+(0,-.76) *+{E}*\frm{-} **@{-};
p+(-.825,.825) **@{-},
(.8,.5) *+{1};
(1.8,.5) *+{2};
(3.2,.5) *+{3};
(4.2,.5) *+{4};
(1.3,-.5) *+{9};
(2.3,-.5) *+{8};
(3.2,-.5) *+{7};
(4.2,-.5) *+{6};
(5.3,-.5) *+{5};
\end{xy}
\]
\caption{Kschischang et~al.'s factor graph  \cite[Fig. 1]{Kschischang01factorgraphs}.  Each variable $x_i$ corresponds to the Boolean cogate tensor $AE^d$ where $d\in\{1,2\}$ is the degree and each factor A-E to a gate tensor with arbitrary coefficients.   Choosing a root vertex yields a factorization of the tensor contraction for efficiently computing the value of the open circuit created by omitting that vertex.}
\label{fig:Kschischang} 
\end{figure}
Let $\Gamma=(V,C,E)$ be a tree.  Then the sum-product algorithm is the factorization of the total and partial contractions achieved by rooting at a vertex and factoring accordingly.  Explicitly, we write the tree in Cayley format in terms of nested angle brackets, keeping the variables (cogates) on the left and tensoring over the leaves at each level.  For Kschischang et~al.'s example factor graph in Figure \ref{fig:Kschischang}, the total contraction factored by rooting at $x_1$ is the scalar
\[
\val(\Gamma)=\langle x_1, A \ot \langle x_2, B \ot \langle x_3, C \ot \langle x_4, D \rangle \ot \langle x_5, E \rangle \rangle \rangle \rangle.
\]
Partial contraction of all but $x_1$ yields a marginal function of $x_1$ as the following element of $V^{1} \ot V^{9}$:
\[
\val(\Gamma')=A \ot \langle x_2, B \ot \langle x_3, C \ot \langle x_4, D \rangle \ot \langle x_5, E \rangle \rangle \rangle.
\]
This is  a tensor-valued partial contraction, the value of the open circuit $\Gamma'$ formed by removing the node labeled $x_1$ from $\Gamma$, so that edges 1 and 9 are dangling.

The sum-product algorithm, particularly in its junction-tree \cite{lauritzen1988local} guise,  is useful in computing the value of circuits with small treewidth.  It has been used to simulate quantum computations \cite{Markov2008}.  Note that one strength of the sum-product algorithm over holographic algorithms is that it requires the tensor coefficients only to lie in a semiring, whereas holographic algorithms require ring coefficients as they exploit cancellation.  We will limit ourselves further to coefficients in a field, typically $\C$.

\section{Pfaffian circuits and their evaluation}\label{sec:HoloCircuits}

In \cite{MR2120307,QCtcbSiPT,ValiantFOCS2004} L. Valiant introduced a  different approach to the same sort of problem (an exponentially large sum of products diagrammed by a factor graph).  Using  
{\it matchgates} and  {\it holographic algorithms}, he proved the existence of polynomial time algorithms for  counting and sum-of-products problems that na\"\i vely appear to have  exponential complexity.
Such algorithms have been studied in depth and further developed  by J. Cai et al. \cite{MR2277247,MR2402465,MR2424719}.
In particular they are able to provide exact, polynomial time solutions to problems with non-tree graphs but are subject to other, somewhat opaque algebraic restrictions.  Holographic algorithms apply if the coordinates of the predicates involved, expressed as tensors, satisfy a collection of polynomial equations called {\it matchgate identities}, generally after a change of basis on the underlying vector spaces (one space per edge). Currently the same change of basis has always been applied to each edge, so simply allowing non-global basis change  may provide additional power.

Each edge of a closed tensor circuit represents a pairing of a vector space and its dual.  Performing a change of basis $B$ along an edge of a tensor circuit means applying the transformation $B$ to the primal and $B^\vee$ to the dual, which does not affect the value of the circuit (Proposition \ref{prop:basischangepairing}).  We may view the result as a new circuit in the standard basis, with the same graph but different predicates, and the same value as the original circuit.  This provides an alternative method to find {\em holographic} reductions \cite{ValiantFOCS2004} from one problem to another: rather than reason about gadgets and changes to the graph, one can just write down the desired circuit, perform a change of basis on each edge, and obtain a new circuit which has exactly the same value but a different set of predicates.  

One reason to perform changes of basis is that the new set may often have a sparsity pattern, or satisfy other conditions that make evaluation easier.  For example, it might now be possible to separate an arity-four predicates into two arity-two predicates, allowing a better factorization.  The improvement of interest to us, however, is that the new predicates can sometimes be made to satisfy the matchgate identities.  When this is possible, and the graph is planar, circuit evaluation can be accomplished very efficiently.

A Pfaffian circuit, open or closed, is a tensor contraction circuit: a collection of Pfaffian gates and cogates arranged in a bipartite graph, e.g. to represent a Boolean satisfiability problem or a circuit computing a function.  There are restrictions on both the predicates allowed (which must be Pfaffian) and the ways in which they may be composed (planarity and matching bases).   Because of the relationship of Pfaffian circuits to Pfaffians, their closed circuit values (full tensor contractions) can be computed efficiently.

\subsection{Pfaffians and subPfaffians}
The Pfaffian of an  $n \times n$ skew-symmetric matrix $\Xi$ is zero if $n$ is odd, one if $n=0$, and for $n>0$ even is the quantity 
\[
\Pf(\Xi) = \sum_{\sigma}\sign(\pi) \xi_{\sigma(1),\sigma(2)}\xi_{\sigma(3),\sigma(4)} \cdots \xi_{\sigma({n-1}),\sigma(n)}
\]
where the sum is over permutations where $\sigma_1<\sigma_2,\sigma_3<\sigma_3, \dots, \sigma_{n-1}<\sigma_n$, and $\sigma_1<\sigma_3<\sigma_5 <\cdots \sigma_{n-1}$. Careful consideration of the special structure of a Pfaffian corresponding to a planar graph can yield an $O(n^{\omega_p})$ algorithm, where the order of planar Pfaffian evaluation $\omega_p$ depends on the the ring and is generally between $1.19$ and $3$; see the discussion after Algorithm \ref{alg:hologeval}.  
If $\Xi$ is empty, $2 \by 2$,  $4 \by 4$, or  $6 \by 6$, the Pfaffian is 1, $\xi_{12}$, $\xi_{12}\xi_{34} - \xi_{13}\xi_{24}+\xi_{23}\xi_{14}$, and $\xi_{12}\Pf \Xi_{3456} - \xi_{13}\Pf \Xi_{2456} +\xi_{14}\Pf \Xi_{2356} - \xi_{15}\Pf \Xi_{2346} + \xi_{16}\Pf \Xi_{2345} $ respectively.

\begin{defn}
For $n \times n$ skew-symmetric matrices $\Xi$ and $\Theta$, define the tensors
\begin{equation} \label{def:eq:sPf1}
\sPf(\Xi)
= \sum_{I \subset[n]} \Pf( \Xi_I )| I \ket
\end{equation}
and 
\begin{equation} \label{def:eq:sPf2}
\sPfd(\Theta)
= \sum_{J \subset[n]} \Pf( \Theta_{J^C} )\bra J | 
\end{equation}
\end{defn}
\noindent where $|I\ket$ denotes the bitstring tensor which is the indicator function for $I$ and $\Xi_I$ is the submatrix of $\Xi$ including only the rows and columns in $I$. 
These easy-to-represent $2^n$ dimensional tensors are the building blocks of holographic computations.  We can also define versions of (\ref{def:eq:sPf1}) and  (\ref{def:eq:sPf2}) in which the tensor bits, and corresponding rows and columns of the matrix, are explicitly labelled as in (\ref{eq:labelledbraket}).

\begin{defn} \label{defn:holographicpredicate}
A gate $G$ (resp. cogate $J$) with $n$ edges is {\em Pfaffian} over a field $\FF$ if there exists an $n \times n$ skew-symmetric  matrix $\Xi$ (resp. $\Theta$) over $\FF$ and a field element $\alpha$ (resp. $\beta$) such that $G=\alpha \sPf \Xi$ (resp.  $J=\beta \sPfd \Theta$).
\end{defn}
 Note that for nonsymmetric predicates, permuting the edges does not affect whether or not the predicate is Pfaffian.  Which edges correspond to which bits of the gate is part of the data of the circuit.

Now we consider how to form the direct sum of labelled matrices and the tensor product of labelled tensors.  Consider  a set $\mcM$ of matrices, each of which has its rows and columns labeled by a different subset of the numbers $[n]=\{1, \dots, n\}$, such that the label sets partition $[n]$.  Suppose $\Xi, \Xi' \in \mcM$ with label sets $I,J \subset [n], I \cap J = \emptyset$.  Define $\Xi \oplus \Xi'$ to be the direct sum of $\Xi$ and $\Xi'$, i.e. the matrix $\Xi''$ with label set $I \union J$ which has $\xi''_{k\ell}=0$ if one of $k,\ell$ is in $I$ and one in $J$, and otherwise the entry is whatever it would have been in $\Xi$ or $\Xi'$, e.g. $\xi''_{k\ell}=\xi'_{k\ell}$ if $\{k,\ell\}\subset I$.   
\[
\bordermatrix{
&2& 5 &7\cr
2&\xi_{22}&\xi_{25}&\xi_{27}\cr
5&\xi_{52}&\xi_{55}&\xi_{57}\cr
7&\xi_{72}&\xi_{75}&\xi_{77}\cr
}\oplus
\bordermatrix{
&3& 4 &9\cr
3&\xi_{33}&\xi_{34}&\xi_{39}\cr
4&\xi_{43}&\xi_{44}&\xi_{49}\cr
9&\xi_{93}&\xi_{94}&\xi_{99}\cr
} = 
\bordermatrix{
&2&3&4&5&7&9\cr
2&\xi_{22}&0&0&\xi_{25}&\xi_{27}&0\cr
3&0&\xi_{33}&\xi_{34}&0&0&\xi_{39}\cr
4&0&\xi_{43}&\xi_{44}&0&0&\xi_{49}\cr
5&\xi_{52}&0&0&\xi_{55}&\xi_{57}&0\cr
7&\xi_{72}&0&0&\xi_{75}&\xi_{77}&0\cr
9&0&\xi_{93}&\xi_{94}&0&0&\xi_{99}\cr
}
\]
Similarly if $T=\alpha |x_2 x_5 x_7 \ket$ and $T'=\beta |x_3 x_4 x_{9} \ket$ are two tensors with labeled bits, define $T \ot T'$ to be $\alpha\beta |x_2 x_3 x_4 x_5 x_7  x_{9} \ket$ and extend linearly to general tensors.

Now suppose $\Gamma = (\text{Gates},\text{Cogates}, \text{Edges})$ is a combinatorial connected planar bipartite graph with $n$ edges. 
A graph is bipartite iff it does not contain an odd cycle.  Thus all of the regions defined by a planar embedding of a planar bipartite graph have an even number of edges.  Therefore every vertex in the dual of $\Gamma$ has even degree, and there exists a non-self-intersecting Eulerian cycle in the dual.  This cycle corresponds to a oriented closed planar curve through $\Gamma$ crossing  each edge of $\Gamma$ exactly once; oriented, it defines an ordering or labelling of the edges by $\{1, \dots, n\}$ where the $r$th edge crossed gets the label $r$.  The corresponding closed curve separates the cogates from the gates, and induces a clockwise or counterclockwise cyclic (but not necessarily consecutive) edge order on the edges incident on any vertex.  Both embedding and curve can be computed in $O(n)$ time.  A proof that such a curve produces a ``valid'' edge order appears in \cite{HAWMG}.  We now give a new proof of this result, 
specializing to a particular curve in order to simplify the argument. 

 Given $\Gamma$ and a planar embedding, define a new graph $\Gamma'$ whose vertices are the gates only.  If $m$ gates are incident to an interior face of $\Gamma$, include the edges connecting them in a $m$-cycle around the region. See Figure \ref{fig:recursivecurve} for an illustration.  The resulting graph $\Gamma'$ is planar; take a spanning tree.  Choose a single exterior gate and connect it with a line to a circle drawn in the plane to contain $\Gamma'$.  Now thicken the edges, circle, connecting line, and vertices representing the gates, and orient the boundary of the resulting region.  This defines a closed curve crossing each edge of $\Gamma$ exactly once.  Call the result a spanning tree order on the edges of $\Gamma$.

We claim that such a curve and order could also have arisen if we had begun with the cogates instead and performed the same construction.  By construction, the gate side of the curve is connected, and so is its complement in the plane.  In other words,  any cycle among the cogates is contractable to a path in the spanning tree.  Were it not, then we would have a gate which is isolated from the rest.  The exterior cycle among the cogates is prevented by the connection to the exterior circle.

The curve construction is depicted in Figure \ref{fig:recursivecurve}.  Another feature of this particular curve is that if a gate node is replaced by a compound gate, we can just invaginate the curve between two edges of the gate and extend the construction recursively (Figure \ref{fig:recursivecurve2}).

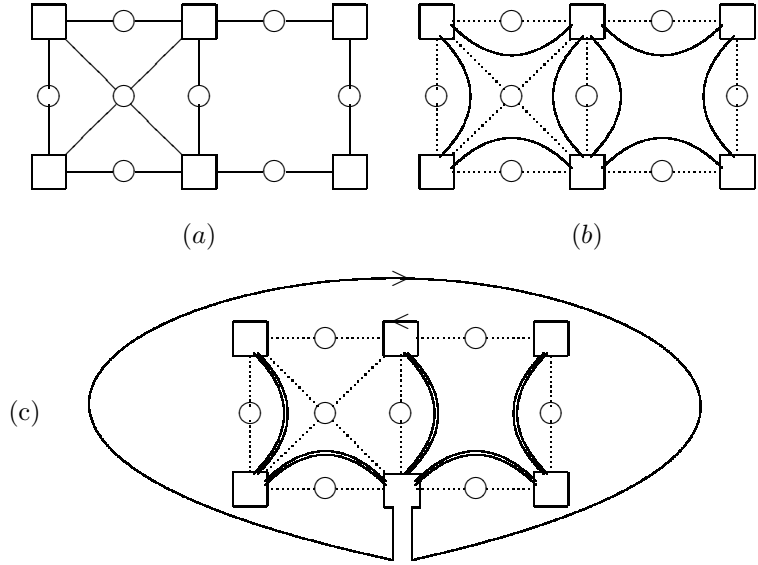
\begin{figure}

\[
\begin{array}{ccc}
\begin{xy}<10mm,0mm>:
(0,0) *+{\phantom{I}\,}*\frm{-};
  p+(1,-1) *+{\;}*\frm{o} **@{-},
p+(1,0) *+{\;}*\frm{o} **@{-};
p+(1,0) *+{\phantom{I}\,}*\frm{-} **@{-};
  p+(0,-.83)  **@{-},
  p+(-.9,-.9)  **@{-},
p+(1,0) *+{\;}*\frm{o} **@{-};
p+(1,0) *+{\phantom{I}\,}*\frm{-} **@{-};
(0,-.25);
p+(0,-.75) *+{\;}*\frm{o} **@{-};
p+(0,-1) *+{\phantom{I}\,}*\frm{-} **@{-};
  p+(.9,.9)  **@{-},
p+(1,0) *+{\;}*\frm{o} **@{-};
p+(1,0) *+{\phantom{I}\,}*\frm{-} **@{-};
  p+(-.9,.9)  **@{-},
p+(0,1) *+{\;}*\frm{o} **@{-},
p+(1,0) *+{\;}*\frm{o} **@{-};
p+(1,0) *+{\phantom{I}\,}*\frm{-} **@{-};
p+(0,1) *+{\;}*\frm{o} **@{-};
 p+(0,.75)  **@{-};
\end{xy}
&&
\begin{xy}<10mm,0mm>:
(0,0) *+{\phantom{I}\,}*\frm{-};
  p+(1,-1) *+{\;}*\frm{o} **@{.},
p+(1,0) *+{\;}*\frm{o} **@{.};
p+(1,0) *+{\phantom{I}\,}*\frm{-} **@{.};
  p+(0,-.83)  **@{.},
  p+(-.9,-.9)  **@{.},
p+(1,0) *+{\;}*\frm{o} **@{.};
p+(1,0) *+{\phantom{I}\,}*\frm{-} **@{.};
(0,-.25);
p+(0,-.75) *+{\;}*\frm{o} **@{.};
p+(0,-1) *+{\phantom{I}\,}*\frm{-} **@{.};
  p+(.9,.9)  **@{.},
p+(1,0) *+{\;}*\frm{o} **@{.};
p+(1,0) *+{\phantom{I}\,}*\frm{-} **@{.};
  p+(-.9,.9)  **@{.},
p+(0,1) *+{\;}*\frm{o} **@{.},
p+(1,0) *+{\;}*\frm{o} **@{.};
p+(1,0) *+{\phantom{I}\,}*\frm{-} **@{.};
p+(0,1) *+{\;}*\frm{o} **@{.};
 p+(0,.75)  **@{.};
(.2,-.05);
  p+(1.6,0) **\crv{~*=\dir{.}p+(.8,-.8)},
(.05,-.2);
  p+(0,-1.6) **\crv{~*=\dir{.}p+(.8,-.8)},
(2.2,-.05);
  p+(1.6,0) **\crv{~*=\dir{.}p+(.8,-.8)},
(2.05,-.2);
  p+(0,-1.6) **\crv{~*=\dir{.}p+(.8,-.8)},
(1.95,-.2);
  p+(0,-1.6) **\crv{~*=\dir{.}p+(-.8,-.8)},
(3.95,-.2);
  p+(0,-1.6) **\crv{~*=\dir{.}p+(-.8,-.8)},
(.2,-1.95);
  p+(1.6,0) **\crv{~*=\dir{.}p+(.8,.8)},
(2.2,-1.95);
  p+(1.6,0) **\crv{~*=\dir{.}p+(.8,.8)},
\end{xy}\\
&&\\
(a)&&(b)\\
\end{array}
\]
\[
\begin{xy}<10mm,0mm>:
(0,0) *+{\phantom{I}\,}*\frm{-};
  p+(1,-1) *+{\;}*\frm{o} **@{.},
p+(1,0) *+{\;}*\frm{o} **@{.};
p+(1,0) *+{\phantom{I}\,}*\frm{-} **@{.};
  p+(0,-.83)  **@{.},
  p+(-.9,-.9)  **@{.},
p+(1,0) *+{\;}*\frm{o} **@{.};
p+(1,0) *+{\phantom{I}\,}*\frm{-} **@{.};
(0,-.25);
p+(0,-.75) *+{\;}*\frm{o} **@{.};
p+(0,-1) *+{\phantom{I}\,}*\frm{-} **@{.};
  p+(.9,.9)  **@{.},
p+(1,0) *+{\;}*\frm{o} **@{.};
p+(1,0) *+{\phantom{I}\,} **@{.};
  p+(-.9,.9)  **@{.},
p+(0,1) *+{\;}*\frm{o} **@{.},
p+(1,0) *+{\;}*\frm{o} **@{.};
p+(1,0) *+{\phantom{I}\,}*\frm{-} **@{.};
p+(0,1) *+{\;}*\frm{o} **@{.};
 p+(0,.75)  **@{.};
(.05,-.2);
  p+(0,-1.6) **\crv{~*=\dir{.}p+(.8,-.8)},
(.1,-.2);
  p+(0,-1.6) **\crv{~*=\dir{.}p+(.8,-.8)},
(2.05,-.2);
  p+(0,-1.6) **\crv{~*=\dir{.}p+(.8,-.8)},
(2.1,-.2);
  p+(0,-1.6) **\crv{~*=\dir{.}p+(.8,-.8)},
(3.95,-.2);
  p+(0,-1.6) **\crv{~*=\dir{.}p+(-.8,-.8)},
(3.9,-.2);
  p+(0,-1.6) **\crv{~*=\dir{.}p+(-.8,-.8)},
(.2,-1.95);
  p+(1.6,0) **\crv{~*=\dir{.}p+(.8,.8)},
(.2,-1.9);
  p+(1.6,0) **\crv{~*=\dir{.}p+(.8,.8)},
(2.2,-1.95);
  p+(1.6,0) **\crv{~*=\dir{.}p+(.8,.8)},
(2.2,-1.9);
  p+(1.6,0) **\crv{~*=\dir{.}p+(.8,.8)},
(1.9,-2.25);
p+(0,-.7)  **\crv{~*=\dir{.}p+(0,-.7)};
p+(.25,0)  **\crv{~*=\dir{.}p+(-7,1.5)&p+(0,6)&p+(7,1.5)};
p+(0,.7)  **\crv{~*=\dir{.}p+(0,.7)};
(2,.8) *{>};
(2,.22) *{<};
(-3,-1) *+{\text{(c)}};
(1.9,-2.25);
p+(-.12,0) **@{-};
p+(0,.45) **@{-};
p+(.48,0) **@{-};
p+(0,-.45) **@{-};
p+(-.12,0) **@{-};
\end{xy}
\]
\caption{Planar spanning tree curve construction: (a) the circuit, (b) the region graph on its gates, and (c) an oriented spanning tree curve.} \label{fig:recursivecurve}
\end{figure}

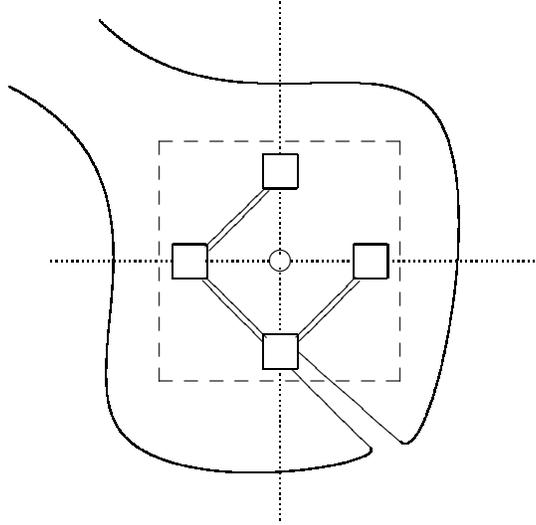
\begin{figure}
\[
\begin{xy}<8mm,0mm>:
(0,0);
p+(4,0) **@{--};
p+(0,-4) **@{--};
p+(-4,0) **@{--};
p+(0,4) **@{--};
(2,-2) *+{\;}*\frm{o};
  p+(-1.5,0) *+{\phantom{I}\,}*\frm{-} **@{.},
  p+(1.5,0) *+{\phantom{I}\,}*\frm{-} **@{.},
  p+(0,1.5) *+{\phantom{I}\,}*\frm{-} **@{.},
  p+(0,-1.5) *+{\phantom{I}\,}*\frm{-} **@{.},
(.2,-2); p+(-2,0) **@{.},
(3.8,-2); p+(2.5,0) **@{.},
(2,-3.8); p+(0,-2.5) **@{.},
(2,-.2); p+(0,2.5) **@{.},
(-1,2);
  p+(5,-7)  **\crv{~*=\dir{.}p+(2,-2)&p+(7.5,1.5)&p+(5.5,-7.5)};
  p+(-1.7,1.5)  **@{-};
  p+(0,.2);
  p+(1,1)  **@{=},
  p+(-.55,0);
  p+(-1,1)  **@{=};
  p+(.05,.5);
  p+(1,1)  **@{=},
(3.5,-5.1);
  p+(-1.3,1.3)  **@{-},
  p+(-6,6)  **\crv{~*=\dir{.}p+(.3,-.3)&p+(-7,-1.5)&p+(-3,4.5)};
\end{xy}
\]
\caption{Recursive view of the planar spanning tree.  The dashed line is a compound gate.} \label{fig:recursivecurve2}
\end{figure}

\begin{lemma}\label{lem:pairwiseotimeoplus}
 Suppose the edge labelling comes from a spanning tree  order. Let $G_1=\sPf \Xi^1$, $G_2=\sPf \Xi^2$ be edge-labelled Pfaffian gates  connected in this spanning tree (and so by a cogate) in a closed circuit $\Gamma$.  $G_1$ and $G_2$ have disjoint label sets $I_1,I_2 \subset [n]$ respectively.   Then $G_1 \ot G_2 = \sPf ( \Xi^1 \oplus  \Xi^2 )$ as an identity of edge-labelled tensors.
\end{lemma}
\begin{proof}
We have 
\begin{eqnarray*}
G_1 \ot G_2 &=&  \sPf(\Xi^1) \ot \sPf(\Xi^2)\\
&=& \left ( \sum_{J_1 \subset I_1} \Pf( \Xi^1_{J_1} ) | J_1 \ket \right )  \left ( \sum_{J_2\subset I_2} \Pf( \Xi^2_{J_2} ) | J_2 \ket \right ) \\
&=& \sum_{J_1 \subset I_1, J_2 \subset I_2} \Pf( \Xi^1_{J_1} )\Pf( \Xi^2_{J_2} ) | J_1 \cup J_2 \ket  \\
&=& \sum_{K \subset I_1 \cup I_2} \Pf( \Xi^1_{K \cap {I_1}} )\Pf( \Xi^2_{K\cap {I_2}} ) | K \ket  \\
\end{eqnarray*}
whereas
\[
\sPf( \Xi^1 \oplus \Xi^2 ) =  \sum_{K \subset I_1 \cup I_2} \Pf( (\Xi^1\oplus \Xi^2)_K ) | K \ket 
\]
so it is sufficient to show that for all $K \subset I_1 \cup I_2$ the coefficients are equal,
\begin{equation} \label{eq:pairwiseotoplus}
\Pf(\Xi^1_{K\cap{I_1}}) \Pf(\Xi^1_{K \cap {I_1}}) = \Pf( (\Xi^1\oplus \Xi^2)_K ).
\end{equation}
This is a slight generalization of the fact  that for skew-symmetric matrices $M,M'$ with rows and columns labelled {\em consecutively} where $1, \dots, k$ label $M$ and $k+1, \dots, \ell$ label $M'$, then 
\begin{equation} \label{eq:blockdiagonal}
\sPf(M) \ot \sPf(M') = \sPf (M \oplus M')
=\sPf
\begin{pmatrix}
M & 0 \cr
0 & M' \cr
\end{pmatrix},
\end{equation}
so we reduce to this case.  If $K$ is completely contained in $I_1$, both sides of Equation (\ref{eq:pairwiseotoplus}) equal $\Pf(\Xi^1_K)$ and similarly if $K\subset I_2$.  Suppose $K$ has nontrivial intersection with both label sets.  If $I_1 \cap K$ is even and $I_2 \cap K$ is odd or vice versa, both sides of  Equation (\ref{eq:pairwiseotoplus})  are zero.  If $I_1 \cap K$ and $I_2 \cap K$ are both odd, the left hand side is zero but we must show the right hand side is also zero.  
But
The matrix  $\Xi^1\oplus \Xi^2$ may be written as a block-diagonal matrix by applying the same permutation to the rows and columns. Thus the left hand side and right hand side of  (\ref{eq:pairwiseotoplus}) are equal up to the sign of this permutation, and the right hand side is also zero.

Now suppose  $I_1 \cap K$ and $I_2 \cap K$ are both even.  We would like to reduce to the block diagonal case (\ref{eq:blockdiagonal}), i.e. the matrix with labelling $i^1_1 ,i^1_2 , \dots ,i^1_{|I_1|}, i^2_1 ,i^2_2 , \dots ,i^2_{|I_1|} $ where $i^1_1 <i^1_2 < \cdots <i^1_{|I_1|}$ and $i^2_1 <i^2_2 < \cdots <i^2_{|I_2|} $.  We claim that the permutation connecting the interleaved curve order $j_1, \dots j_n$ of $\Xi^1\oplus \Xi^2$ to the block diagonal order is even.  This is where we will need the curve order.

Draw a small curve separating the two gates and the cogate connecting them, with the small curve crossing each edge to the outside world (Figure \ref{fig:recursivecurve} (e)).  The big planar spanning tree curve enters and leaves this circle in the same ``wedge'' between two edges, i.e.\ across a face of $\Gamma'$.  Thus we may assume the order induced on the edges of the gates we are considering  is the lower part of the $G_1$ indices in (as consecutive as possible) ascending order, followed by all the $G_2$ indices in ascending (as consecutive as possible) order, followed by the remainder of the $G_1$ indices in order.  Since $G_2$ has an even number of edges, the permutation required to move the top half of the $G_1$ indices past them to join the lower half (so the indices are in block diagonal order) is even. 
\end{proof}

When joining compound gates, the argument is unchanged because of the recursive nature of the spanning tree order (Figure \ref{fig:recursivecurve} (d)).  Thus the Lemma shows that we may join two compound gates and maintain the correct Pfaffian value, and inductively join all the gates.  The critical part of the following result is that the same edge order works for both gates and cogates.
\begin{theorem} \label{thm:otimesoplus}
Let  $\Gamma$ be a connected bipartite embedded planar graph with the edge order induced by a planar  spanning tree curve, and let $\{G_i\}$ be the collection of edge labelled Pfaffian gates, $G_i =\sPf \Xi_i$.  Then $\ot_i G_i= \oplus_i \Xi_i$.  The analogous result holds for cogates with the same curve order.
\end{theorem}

The total contraction of the dual gate and cogate sub-Pfaffian tensors can be computed efficiently.  The following result follows directly from one which appears in \cite[p. 110]{MR1069389}.
\begin{theorem}[Pfaffian kernel theorem] \label{thm:kernel}
Let  $\Theta$ and $\Xi$ be $n \times n$ skew-symmetric matrices with $n$ even. Let $\tilde{\Theta}$ be $\Theta$ with signs flipped in the checkerboard pattern which flips the diagonal.  Then the pairing
\begin{equation} \label{thm:kernel:eq}
\langle \sPfd \Theta, \sPf \Xi \rangle = \Pf( \tilde{\Theta} + \Xi).
\end{equation}
\end{theorem}
On the left hand side is an exponential sum, which can be over all ``inputs'', and on the right a quantity which can be computed with the complexity of Pfaffian evaluation.  This kernel allows us to compute the value of a closed Pfaffian circuit efficiently.  For example under certain conditions, we may evaluate  the sum over all inputs (for \#decision and \#function problems) in a way analogous to a quantum unitary operator acting on a superposed state.  Note that Valiant's Pfaffian Sum Theorem  \cite{QCtcbSiPT} is Theorem \ref{thm:kernel} with $\theta_{i<j}=\lambda_i \lambda_j$. 

A circuit was defined in Definition \ref{defn:tensconcirc} for general predicates.  We now add some restrictions to facilitate fast evaluation by Theorems \ref{thm:otimesoplus} and  \ref{thm:kernel}.
Every edge $e$ in $\Gamma$ corresponds to to a vector space $V^e$.  A {\em typed predicate} is a predicate together with a type labeling each edge, which indicates which basis change we make on that edge's vector space.

\begin{defn}
A {\em Pfaffian circuit} is a combinatorial object with 
\begin{itemize}
\item[(i)] a set $\scrG$ of typed gates and $\scrJ$ of typed cogates with coefficients in a field $\FF$, the gadgets we may use to build the circuit, and
\item[(ii)] a connected bipartite graph $\Gamma = (\text{Gates}, \text{Cogates}, \text{Edges})$ where each gate in Gates is drawn from $\scrG$ and each cogate from $\scrJ$, and where the gate and cogate incident on each edge agree on the edge type.  This represents the circuit architecture: which vector spaces are paired, including the assignment for non-symmetric predicates,
\end{itemize}
such that every typed gate and typed cogate is Pfaffian after the indicated change of basis, and $\Gamma$ is planar.
\end{defn}

The value of a closed Pfaffian circuit can be computed efficiently.

\begin{theorem}
The complexity of computing the value of a closed Pfaffian circuit over a field $\FF$ with graph $\Gamma$ is the complexity $O(n^{\omega_p})$ of computing a planar Pfaffian with graph $\Gamma$  over $\FF$.
\end{theorem}
\begin{proof}
Apply Theorems \ref{thm:otimesoplus} and \ref{thm:kernel}.  One must accumulate in a quantity $\gamma(\text{Gates},$ $\text{Cogates})$ the adjustment factors $\alpha$ and $\beta$ of Definition \ref{defn:holographicpredicate}, which thus depend only on the number and type of gates and cogates appearing in the closed Pfaffian circuit.  An explicit algorithm to obtain the value of a circuit is given in Algorithm \ref{alg:hologeval}; its complexity is dominated by evaluating the Pfaffian.  

\end{proof}

\begin{algorithm} \label{alg:hologeval}
$\;$\\
\noindent {\bf Input: } A closed Pfaffian circuit $\Gamma$ with an even number $n$ of edges.\\
\noindent {\bf Output: } The value $\val(\Gamma)$ of the circuit.
\smallskip
\begin{enumerate}
\item  If not given, find a planar embedding of $\Gamma$ ($O(n)$ time \cite{boyer_cutting_2004,hopcroft_efficient_1974,boyer_stop_2004}), and define a planar spanning tree edge order ($O(n)$).
\item Form $\Omega$, a skew-symmetric $n \times n$ matrix of zeros as follows.  Set $\gamma=1$.
  \begin{enumerate}
  \item For each Pfaffian gate $G$ with changes of basis transformation $T_G$ and $T_G(G)=\alpha_G \sPf(\Xi_G)$ (after basis change) with edges $I=\{i_1, \dots, i_d\}$ incident, let $\gamma \leftarrow \alpha_G\gamma$ and set the submatrix $\Omega_{I}=\Xi_G$.  
  \item For each Pfaffian cogate $J$ with changes of basis transformation $T_J$ and $T_J(J)=\beta_J \sPfd(\Theta_J)$  with edges $I=\{i_1, \dots, i_d\}$ incident, let $\gamma \leftarrow \beta_J\gamma$ and set $\Omega_{I}=\tilde{\Theta}_J$ with $\tilde{\theta}_{ij} = (-1)^{i+j+1} \theta_{ij}$.
  \end{enumerate}
\item  Output $\gamma  \Pf(\Omega)$.  $O(n^{\omega_p})$, the order of Pfaffian evaluation.
\end{enumerate}
\end{algorithm}

\smallskip

We have used $\omega_p$ for the order of Pfaffian evaluation because it can depend on the ring; we now collect a few results on what is known about $\omega_p$.  The evaluation of the Pfaffian can be accomplished without division over an integral domain in $O(n^3)$ \cite{Galbiati91}.  One can also exploit the planar sparsity structure of the graph to get a faster evaluation.  Avoiding writing down $\Omega$ explicitly, and using a nested dissection algorithm \cite{lipton_generalized_1979} for $\Pf(\Omega)$, the evaluation can be obtained up to sign in $O(n^{\omega/2})$, where $\omega$ is the exponent of matrix multiplication.  Over finite fields this smaller exponent can be achieved with randomization \cite{yuster2008matrix}.  A discussion of numerical issues can be found in Section 12 of \cite{ValiantFOCS2004}.

After all changes of basis in a closed circuit, we are left with a circuit in which all edges are in the standard basis.  The change of basis device is just a means of making non-obvious transformations of predicates with which it is natural to express the problem of interest (e.g. Boolean functions) into predicates which are generally non-Boolean, but whose composition represents the same counting problem.  Thus in the next section 
we first describe Pfaffian predicates in standard basis, where all holographic computations end up before the contraction is performed. 

\section{Pfaffian predicates in the standard basis} \label{sec:CBholographic}

In this section we study which predicates can be implemented as the subPfaffian of some matrix, or small contractions of a few predicates which can, and relate these to {\em spinor varieties}.  We first study the easiest case of Pfaffian gates: gates of even-weight support that are subPfaffians, initially described in \cite{HAWMG}.  This corresponds to the restriction of the even spinor to an affine open.  Then we study the affine odd case, and finally the homogenized, projective versions.  The first two are easy to characterize in terms of the image of a natural parameterization by subPfaffians and in terms of defining equations, and represent set-theoretic affine complete intersections.  The projective versions will require additional equations.  The results in this section mirror those for matchgate signatures in the standard basis (e.g. \cite{caisigpaper}).

\subsection{Affine spinor varieties and matchgates}
It is easy to characterize which tensors can be written as the subPfaffian tensor $\sPf \Xi$ of some skew-symmetric matrix $\Xi$.  For a ${0 \choose n}$ tensor $T$ and an $n$-bit string $x$, denote by $T_{| x \ket} = \la x | ( T) $ the coefficient of $| x \ket$ in the expansion of $T$ in the standard basis.  Then the equations a degree $n$ gate must satisfy to be the subPfaffian tensor of some matrix are
\begin{enumerate}
\item  {\bf Parity:} $T_{| x \ket} =0 $ whenever there are an odd number of ones in the bitstring $ x $ ($2^n-1$ equations), and 
\item  {\bf Consistency:} if $n\geq 4$, there are consistency conditions imposed by larger Pfaffians being writable in terms of smaller Pfaffians.  It is enough to consider all the Pfaffians in terms of the individual matrix entries.  For example for  a ${0 \choose 4}$ gate $G$, this requires that
\[
G_{|0000\ket} G_{|1111\ket} = 1 G_{|1111\ket} =  G_{|1100\ket} G_{|0011\ket} -  G_{|1010\ket} G_{|0101\ket} + G_{|0110\ket} G_{|1001\ket}
\]
\end{enumerate}
where $G_{|0000\ket}=1$ since the Pfaffian of the empty matrix is one. We will see shortly that (1) may be exchanged for the complementary requirement that even weight coefficients must be zero by considering a partial contraction of an $(n+1)$-gate and a $1$-cogate, and (2) may likewise be homogenized by partial contraction.

Since in the standard basis one of these parity restrictions will always hold, it is convenient to consider the variety as a subvariety of the vector space $W_+$ of even or $W_-$ of odd weight induced basis elements.  

Together these equations are sometimes called Grassmann-Pl\"ucker relations,  matchgate identities \cite{QCtcbSiPT}, generators of the ideal of a spinor variety, or defining equations of the (even) orthogonal Grassmannian \cite{mukai_curves_1995}. They also arise naturally in connection with Clifford algebras.  
The space of Pfaffian (affine with even support) gates is the image of the polynomial map $\sPf(\Xi)$ as $\Xi$ varies over all skew-symmetric matrices in $\FF$ and so has the structure of an algebraic variety.  The requirements are the same for the dual with the obvious modifications (support on bitstrings with an even number of zeros, etc.).  Taking the Zariski closure in $\PP^{2^n-1}$, or $\PP W_+$ or $\PP W_-$ as appropriate and homogenizing will lead to the projective version.

We now connect these requirements to the original matchgate identities  \cite{QCtcbSiPT}.  
Valiant's  \cite{QCtcbSiPT} two-input, two-output gates act independently on the parity subspaces and the determinant of both must be the same.  To see this, consider a ${0 \choose 4}$ gate $G$ interpreted as two-input, two-output gate $G$.
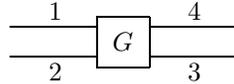
\[
\begin{xy}<10mm,0mm>:
(0,0) *++{G}*\frm{-};
(-1.5,.2);  
  p+(1.15,0) **@{-},
  p+(.6,.2) *+{1},
(-1.5,-.2);  
  p+(1.15,0) **@{-},
  p+(.6,-.2) *+{2},
(.35,.2);  
  p+(1.15,0) **@{-},
  p+(.6,.2) *+{4},
(.35,-.2);  
  p+(1.15,0) **@{-},
  p+(.6,-.2) *+{3},
\end{xy}
\] 
Writing out the linear transformation it represents (as one might for a unitary operator in quantum computing), we get a matrix in the ``operator'' ordering, viewing bit $4$ as the output corresponding to input bit $1$ and bit $3$ as the output corresponding to input bit $2$, which is
\[
\bordermatrix{
&| 0_10_2 \ket & | 0_11_2 \ket & | 1_10_2 \ket & | 1_11_2 \ket \cr
| 0_40_3 \ket  & b_{11} &  b_{12} &  b_{13} &  b_{14} \cr
| 0_41_3 \ket  &  b_{21} &  b_{22} &  b_{23} &  b_{24} \cr
| 1_40_3 \ket  &  b_{31} &  b_{32} &  b_{33} &  b_{34} \cr
| 1_41_3 \ket  &  b_{41} &  b_{42} &  b_{43} &  b_{44} \cr
}.
\]
In the cyclic ordering coming from the planar curve of Theorem \ref{thm:otimesoplus}, the matrix is 
\[
\bordermatrix{
&| 0_10_2 \ket & | 0_11_2 \ket & | 1_10_2 \ket & | 1_11_2 \ket \cr
| 0_30_4 \ket  & b_{11} &  b_{12} &  b_{13} &  b_{14} \cr
| 0_31_4 \ket  &  b_{31} &  b_{32} &  b_{33} &  b_{34} \cr
| 1_30_4 \ket  &  b_{21} &  b_{22} &  b_{23} &  b_{24} \cr
| 1_31_4 \ket  &  b_{41} &  b_{42} &  b_{43} &  b_{44} \cr
}.
\]
Valiant's matchgate equations are 
\begin{enumerate}
\item {\bf Parity:}  $b_{12} \!= \! b_{13}\! =\! b_{21}\! =\! b_{24}\! =\! b_{31}\! =\! b_{34}\! =\! b_{42}\! =\! b_{43}\! = \!0$, i.e. the coefficient of any four-bit $|x\ket$ where $x$ has an odd number of ones must be zero, and 
\item {\bf Consistency:} $b_{11}b_{44} - b_{22}b_{33} = b_{14}b_{41} - b_{23}b_{32}$.  
\end{enumerate}
In other words, the gate must have the form, in the operator ordering,  
\[
\bordermatrix{
&| 0_10_2 \ket & | 0_11_2 \ket & | 1_10_2 \ket & | 1_11_2 \ket \cr
| 0_40_3 \ket  & b_{11} &  0 &  0 &  b_{14} \cr
| 0_41_3 \ket  &  0 &  b_{22} &  b_{23} &  0 \cr
| 1_40_3 \ket  &  0 &  b_{32} &  b_{33} &  0 \cr
| 1_41_3 \ket  &  b_{41} & 0 & 0 &  b_{44} \cr
}
\]
and the matrices $$B_{even}=\begin{pmatrix}b_{11} & b_{14}\cr b_{41} & b_{44}\end{pmatrix}$$ acting on the even parity subspace and $$B_{odd}=\begin{pmatrix}b_{22} & b_{23}\cr b_{32} & b_{33}\end{pmatrix}$$ acting on the odd must have the same determinant, $\det B_{even}=\det B_{odd}$.  When this holds and $b_{11}=1$, so $b_{44}=b_{14}b_{41}-b_{23}b_{32}+b_{22}b_{33}$, we can write $G$ as a subPfaffian (suppressing the below-diagonal entries):
\[
G=\sPf \left ( 
\bordermatrix{
&1&2&3&4\cr
1&0&b_{14}&b_{23}&b_{33}\cr
2&&0&b_{22}&b_{32}\cr
3&&&0&b_{41}\cr
4&&&&0\cr
}
\right ).
\]
So $G=|0_1 0_2 0_3 0_4 \ket + b_{14}|1100\ket  + b_{23}|1010\ket  + b_{33}|1001\ket  + b_{22}|0110\ket  + b_{32}|0101\ket  + b_{41}|0011\ket  + b_{44}|1111\ket $ as desired.  When $b_{11} \neq 1$, we may multiply $\sPf \Xi$ by a constant $\alpha$ that will be accumulated in $\gamma$.  When $b_{11} =0$, we will need to perform a contraction.

\subsection{Homogenization, odd support, and projective spinor varieties}\label{ssec:homogodd}
We now extend to (1) affine spinor varieties supported on the odd parity subspace of $\FF^{2^n}$ and (2) the homogenized version where we obtain full projective spinor varieties.

First, note that the support of a predicate is the set of induced basis elements (e.g. $|1101\ket$) with nonzero coefficients.

\begin{proposition} \label{prop:supportparity}
A predicate in the standard basis which is Pfaffian, or a partial contraction of Pfaffian predicates, has support either on the induced basis elements of odd or even weight.
\end{proposition}
\begin{proof}
First, note that the contraction of one simple predicate against another, where one is completely contracted, follows the usual parity rules: an odd and an odd or even and even make even, and mixed odd and even make odd.  
A single gate or cogate, not partially contracted, has pure parity.  Then we proceed by induction on the number of predicates.
\end{proof}
This corresponds to the odd ${S}_-$ and even  ${S}_+$ spinor varieties. 
The odd and even $\Spin_{2n}$ varieties \cite[p.390]{FH} each have dimension ${n \choose 2}$, and lie in either the $\PP^{2^{n-1}-1}$ corresponding to the projectivization of the even weight vectors or the odd weight vectors.  Note that in our dehomogenized version, where we consider the image of $\sPf \Xi$ in the affine open where the coefficient of $|00 \cdots 0\ket$ is one, the ${n \choose 2}$ dimensions correspond directly to the ${n \choose 2}$ free parameters of $\Xi$.  We now show how to explicitly parameterize these odd and even varieties, and their homogeneous versions (where the coefficient of $|00 \cdots 0\ket$ may be zero).  The parity switch to the odd variety and the homogenization can be accomplished by partial contraction.  
Though obscured by change of basis, this division remains when we consider predicates with basis change.  For example, no basis change can make $|00\ket + |11\ket$ into the $2|00\ket -|01\ket -|10\ket$ needed for $X$-matchings, but this is easily done reducing to $|01\ket + |10\ket$ (see Example  \ref{ex:Xmatchings}).  Adding one edge and the arity-one dual predicate to a predicate flips its parity in the sense of Proposition \ref{prop:supportparity}. 

In what follows we ignore a scaling factor that can be absorbed in the constant $\alpha$ or $\beta$ by which a subPfaffian vector is multiplied and which is accumulated in $\gamma$.  
We first observe that parity is the only restriction up to arity three.

\begin{proposition} \label{prop:aritythreeparity}
A gate of arity at most three can be realized in the standard basis by partial contraction of Pfaffian predicates if and only if it is supported on even or odd weight.
\end{proposition}
\begin{proof}
We describe the situation for gates; the cogate version is completely analogous.  The only if direction is Proposition \ref{prop:supportparity}.  
 In the standard basis, the unique (up to scaling) arity-one simple gate is $|0\ket$, the subPfaffian tensor of the $1\by1$ zero matrix.  This is the even weight piece; the odd weight gate $|1\ket$ can be obtained as the partial contraction $\sPf \Xi (\bra 1 |)$ for the $2\by 2$ matrix $\Xi$ with free parameter $\xi_{12}=1$:
\[
\val \left (
\begin{xy}<10mm,0mm>:
(2,0)  *+{0}*\frm{-};
  p+(-.8,0)  **@{-},
\end{xy}
\;
\right ) = |0\ket
\qquad
\text{and}
\qquad
\val \left (
\begin{xy}<10mm,0mm>:
(6,0)  *+{\Xi}*\frm{-};
  p+(-.8,0)  **@{-},
  p+(.8,0)  *+{1}*\frm{o}  **@{-},
\end{xy}
\right ) = |1\ket.
\]
Similarly, the $ONE^1$ junction is Pfaffian and $ZERO^1$ is odd Pfaffian in the standard basis:
\[
\val \left (
\begin{xy}<10mm,0mm>:
(2,0)  *+{1}*\frm{o};
  p+(-.8,0)  **@{-},
\end{xy}
\;
\right ) = \bra 1|
\qquad
\text{and}
\qquad
\val \left (
\begin{xy}<10mm,0mm>:
(6,0)  *+{0}*\frm{-};
  p+(-.8,0) *+{=}*\frm{o}  **@{-};
  p+(-.8,0)  **@{-},
\end{xy}
\right ) = \bra0|.
\]
Linear combinations  $x |0\ket + y |1\ket$ where both $x$ and $y$ are nonzero are impossible (without change of basis).

For arity two gates,  Proposition \ref{prop:supportparity} tells us we should get support  $\{|00\ket, |11 \ket\}$ or $\{|01\ket, |10\ket\}$;  the first comes from pure gates and the second from gates with arity three contracted against the arity-one junction.
\[
\val \left (
\;
\begin{xy}<10mm,0mm>:
(2,0)  *+{\Xi}*\frm{-};
  {p+(-.21,.15);p+(-.6,0)  **@{-}},
  {p+(-.21,-.15);p+(-.6,0)  **@{-}},
\end{xy}
\;
\right ) = |00\ket + \xi_{12}|11\ket
\qquad
\val \left (
\;
\begin{xy}<10mm,0mm>:
(6,0)  *+{\Xi}*\frm{-};
  {p+(-.21,.15);p+(-.6,0)  **@{-}},
  {p+(-.21,-.15);p+(-.6,0)  **@{-}},
  p+(1,0)  *+{1}*\frm{o}  **@{-},
\end{xy}
\,
\right ) = \xi_{23}|01\ket + \xi_{13}|10\ket
\]

The arity three case is similar.  
The (even) Pfaffian case, the image of $\sPf$ applied to a $3 \by 3$ matrix, gives points of the form $|000\ket + \xi_{12}|110\ket + \xi_{13}|101\ket + \xi_{23}|011\ket$ and scalar multiples where $G_{|000\ket} \neq 0$.  For the odd Pfaffian case, we get arbitrary points, where $\Xi$ is $4\by 4$,
\[
\val \left (
\;
\begin{xy}<10mm,0mm>:
(6,0)  *+{\Xi}*\frm{-};
  {p+(-.21,.15);p+(-.6,0)  **@{-}},
  {p+(-.21,0);p+(-.6,0)  **@{-}},
  {p+(-.21,-.15);p+(-.6,0)  **@{-}},
  p+(1,0)  *+{1}*\frm{o}  **@{-},
\end{xy}
\,
\right ) = \xi_{34}|001\ket +\xi_{24}|010\ket + \xi_{14}|100\ket + \lambda|111\ket
\]
where $\lambda = \Pf \Xi$, and the four coefficients can be chosen independently with no relations (so this fills the odd weight space and there is no need for a homogeneous odd case).  The homogeneous even case is, with a $5\by5$ matrix $\Xi$ and $2\by 2$ matrix $\Theta$ with free parameter $\theta$ so $\sPfd \Theta = \theta\bra00|+\bra 11|$,
\begin{eqnarray*}
\val \left (
\;
\begin{xy}<10mm,0mm>:
(6,0)  *+{\Xi}*\frm{-};
  {p+(-.21,.15);p+(-.5,0)  **@{-}},
  {p+(-.21,0);p+(-.5,0)  **@{-}},
  {p+(-.21,-.15);p+(-.5,0)  **@{-}},
  p+(.7,0)  *+{\theta}*\frm{o},
  {p+(+.21,.05);p+(.3,0)  **@{-}},
  {p+(+.21,-.05);p+(.3,0)  **@{-}},
\end{xy}
\,
\right ) &=& (\theta + \xi_{45})|000\ket +(\theta\xi_{1245}+\xi_{12})|110\ket \\
&& +(\theta\xi_{1345}+\xi_{13})|101\ket +(\theta\xi_{2345}+\xi_{23})|011\ket,
\end{eqnarray*}
Where $\xi_{ijk\ell}:=\Pf\Xi_{ijk\ell}$.  Setting $\xi_{12}=\xi_{13}=\xi_{23}=\xi_{45}=0$ we obtain 
\begin{eqnarray*}
\val(\Gamma)&=& \theta|000\ket 
+(\xi_{15}\xi_{24}-\xi_{14}\xi_{25}+)|110\ket \\
&&+(\xi_{15}\xi_{34}-\xi_{14}\xi_{35})|101\ket 
+(\xi_{24}\xi_{35}+\xi_{25}\xi_{34}-\xi_{14}\xi_{35})|011\ket,
\end{eqnarray*}
an arbitrary point of $\C^4$ varying $\theta$ and the remaining free entries of $\Xi$.  
\end{proof}
Essentially this is because for gates of arity at most three, there are no nontrivial Pfaffians in the subPfaffian vector.  Thus the first case where the variety 
of realizable gates is not completely characterized by the support is arity four, where also a consistency condition must hold.  Here the even parity piece is a ${4 \choose 2}=6$-dimensional variety in the open affine $\C^7\subset \PP^7=\PP(\{|ijk\ket\})$ where $G_{|0000\ket}=\xi_{\emptyset}=1$, so a hypersurface, cut out by the $4\times 4$ Pfaffian quadric
\[
G_{|0000\ket} G_{|1111\ket} -  G_{|1100\ket} G_{|0011\ket} +  G_{|1010\ket} G_{|0101\ket} - G_{|0110\ket} G_{|1001\ket}.
\]
The odd case is similar.

For $n=5$, we again have a set-theoretic affine complete intersection with a 10-dimensional variety in $\C^{15}$ cut out by ${5 \choose 4}=5$ Pfaffian consistency conditions  of the form $\xi_{ijk\ell} = \Pf\Xi_{ijk\ell}$.

The homogenized version is not a complete intersection, in fact this is an important example in connection with the conjecture of Hartshorne that for $n>\frac{2}{3}N$ every nonsingular $n$-dimensional variety in $\PP^N$ is a complete intersection.  The six-dimensional Grassmann variety in $\PP^9$ and this 10-dimensional spinor in $\PP^{15}$ are the only varieties $X$ corresponding to orbits of linear algebraic groups for which $n=\frac{2}{3}N$ where $X$ is not a complete intersection \cite[p. 64]{Zak}.
In this case, the additional five relations needed in the homogeneous case are of the form \cite{mukai_curves_1995}
\[
\xi_{12}\xi_{1345} - \xi_{13}\xi_{1245} + \xi_{14}\xi_{1235} - \xi_{15}\xi_{1234}
\]
and are easily shown to lie in the ideal of the Pfaffian relations when we restrict to the case $G_{|00000\ket}=\xi_{\emptyset}\neq 0$. 

For $n=6$, we have a ${6 \choose 2}=15$-dimensional subvariety of $\PP^{31}$ defined by ${6 \choose 4}+{6 \choose 6}=16$ equations, and so on. 
\begin{proposition}
The set of even gates in the standard basis is a set-theoretic affine complete intersection in the open affine $G_{|00\dots 0\ket} \neq 0$, cut out by the nontrivial Pfaffians.
\end{proposition}
\begin{proof}
The dimension is ${n \choose 2}$ and the dimension of the ambient space is $2^{n-1}-1$, while the number of equations is 
\[
{n \choose 4} + {n \choose 6} + \cdots +{n \choose n-2} +{n \choose n} 
\]
so we must show the identity, if $n$ even 
\[
{n \choose 4} + {n \choose 6} + \cdots +{n \choose n-2} +{n \choose n} = 2^{n-1}-1 - {n \choose 2}
\]
and if $n$ odd, 
\[
{n \choose 4} + {n \choose 6} + \cdots +{n \choose n-1}  = 2^{n-1}-1 - {n \choose 2}
\]
holds.

The first can be rearranged to 
\[
{n \choose 0} + {n \choose 2} + {n \choose 4} + {n \choose 6} + \cdots +{n \choose n-2} +{n \choose n} = 2^{n-1}
\]

\end{proof}

In the matchgrid/matchcircuit approach, implementing a gate of arity $r$ may require $O(r^4)$ edges \cite{valiant2006accidental,cai2007theory}:

\begin{theorem}[5.1 in \cite{valiant2006accidental}, based on 4.1 in \cite{cai2007theory} ]
For any matchgrid with $r$ external nodes there is another matchgrid with the same standard signature that has $O(r^4)$ edges of which $O(r^2)$ have weight different from 1.
\end{theorem}

This means that a matchgate holographic algorithm has complexity $O(r_{max}^4 n^{\omega_p})$, where $r_{max}$ is the arity of the matchgate with the most incident edges.  In contrast, in our approach $r_{max}^4$ is replaced by the constant two.
\begin{theorem} \label{thm:twonewedges}
Any predicate of arity $n$ ($n$ external nodes) which is Pfaffian, or the image of a Pfaffian predicate after basis change, can be implemented with at most two additional edges connecting to one additional predicate that has no external edges.
\end{theorem}
\begin{proof}
The parity-switch and homogenization constructions in the proof of Proposition \ref{prop:aritythreeparity} work as well for gates of arbitrary arity.  Then to implement any gate in the standard basis, we need to add at most one new predicate and at most two new edges.
\end{proof}

It would be interesting to see if this could be combined with lower bounds to strengthen Valiant's result (Theorem 5.2 in \cite{valiant2006accidental}) that there is no elementary matchgrid algorithm for 3CNF by weakening the hypotheses in the condition 'elementary.'

\subsection{Some problems in the standard basis} \label{ssec:stdbasisapplications}
Let us consider some sample problems that can be solved without resorting to a change of basis.  One illustrative class are certain lattice path enumeration problems.  A recent survey is \cite{humphreys2010history}.  
With zero-one weights,  these are purely combinatorial counting problems.  If we place weights on the edges, we can use the fast partition function to do parameter inference in the corresponding statistical model over paths or loops efficiently.
\begin{example}[Lattice paths on a region of a square grid]
Consider the problem of counting paths in a square grid (Figure \ref{fig:grids}).  
The number of monotone (N,E) lattice paths from $(0,0)$ to $(m,n)$ in a $m \times n$ grid is ${m+n \choose m}$.  Suppose more generally we want to count the number of paths using steps in any of the four directions, subject to some restrictions.  For example we might want the path to start somewhere and finish somewhere else, to pass through certain segments, or not to intersect itself.  By representing these conditions in the clauses, we can sometimes compute these quantities with a Pfaffian holographic algorithm in the standard basis.

Suppose we want a satisfying assignment to correspond to a monotone path.  Then we need to (1) set a start and end point, (2) require that at each vertex other than the start and end, the path both enters and leaves and (3) exclude South and West moves.  We can accomplish (1) by adding a $\bra 1|$ junction attached to the desired start and end vertices as shown in Figure \ref{fig:grids}(b).  Requirement (2) is enforced by only allowing, at each gate corresponding to a grid vertex, either the activation of no edges or an even number of edges.  This happens automatically if we use the even parameterization and the standard basis.  Since no grid vertex may have an odd number of active edges incident, the path must grow from  the start and end vertices. 

Since in general $\Pf \Xi \neq 0$, and this is the coefficient of a crossing $+$  of two paths, paths with loops and even isolated loops are generally included (loops are counted only once, with no winding) .  South and east moves are also included.  
\begin{observation}
Given a start vertex  $V_S$  and end $V_E$ vertex weakly northeast of it, monotone paths connecting them are exactly those words $w$ on alphabet $N,S,E,W$ containing no $ES$ or $NW$ subwords, and such that $w V_S  =  V_E$.
\end{observation}

The six possible paths through a vertex correspond directly (Figure \ref{fig:squaregridgates}) to the six entries of each $\Xi$, so to exclude a certain path we need only zero out its entry.  With the labelling of Figure \ref{fig:squaregridgates}, setting $\xi_{34}=0$  
disallows the subwords $ES$ and $NW$ in the path, which with the start and end conditions restricts to monotone paths (and excludes isolated loops).  

For many special cases of lattice path walks, explicit results are available (e.g.\ in terms of a generating function).  For example, the number of monotone paths from $(0,0)$ to $(n,n)$ in an $n \times n$ grid that stay weakly above the  line $y=x$ is the $n$th Catalan number $C_n =(n+1)^{-1} {2n \choose n} $.  The $n$th Schr\"oder number counts such paths but with allowed step $N,E,$ or $NE$ to $(n,n)$, the central Delannoy number counts paths with these steps but without the above-diagonal restriction, and so on.  

In general suppose we are given any connected lattice region consisting of $n$ boxes, together with a start and a finish vertex.  Then we can write down a skew-symmetric matrix with at most $(8n+2)$ rows (with this bound achieved in the case of diagonal boxes), the square root of the determinant of which is the number of monotone paths through the region connecting the start to the finish.  This yields an $O(n^{\omega_p})$ algorithm for performing this count.  The resulting matrix can also be examined to obtain expressions that simplify the Pfaffian.

\begin{figure}[htb]
\[
\begin{xy}<10mm,0mm>:
(0,0) *+{\phantom{I}\,}*\frm{-};
p+(1,0) *+{\;}*\frm{o} **@{-},
p+(0,1) *+{\;}*\frm{o} **@{-},
p+(-1,0) *+{\;}*\frm{o} **@{-},
p+(0,-1) *+{\;}*\frm{o} **@{-},
p+(.6,-.2) *+{2},
p+(.2,.6) *+{1},
p+(-.6,.2) *+{4},
p+(-.2,-.6) *+{3},
(3,0) *+{\phantom{I}\,}*\frm{-};
p+(1,0) *+{\;}*\frm{o} **@{=},
p+(0,1) *+{\;}*\frm{o} **@{=},
p+(-1,0) *+{\;}*\frm{o} **@{-},
p+(0,-1) *+{\;}*\frm{o} **@{-},
p+(.7,.7) *+{\xi_{12}};
(6,0) *+{\phantom{I}\,}*\frm{-};
p+(1,0) *+{\;}*\frm{o} **@{-},
p+(0,1) *+{\;}*\frm{o} **@{-},
p+(-1,0) *+{\;}*\frm{o} **@{=},
p+(0,-1) *+{\;}*\frm{o} **@{=},
p+(-.7,-.7) *+{\xi_{34}};
(9,0) *+{\phantom{I}\,}*\frm{-};
p+(1,0) *+{\;}*\frm{o} **@{=},
p+(0,1) *+{\;}*\frm{o} **@{=},
p+(-1,0) *+{\;}*\frm{o} **@{=},
p+(0,-1) *+{\;}*\frm{o} **@{=},
p+(.7,.7) *+{\Pf\Xi};
\end{xy}
\]
\caption{The ${4 \choose 2}=6$ paths through a vertex correspond to the six free parameters of a $4 \times 4$ skew-symmetric matrix $\Xi$, with the crossing of two paths through the vertex corresponding to $\Pf \Xi$. } \label{fig:squaregridgates}
\end{figure}
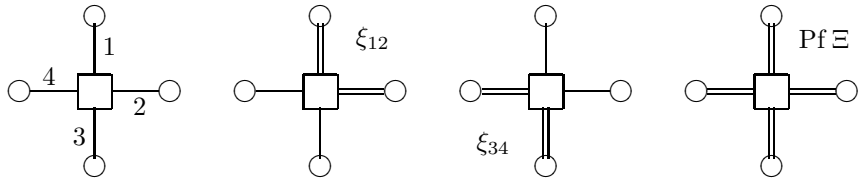

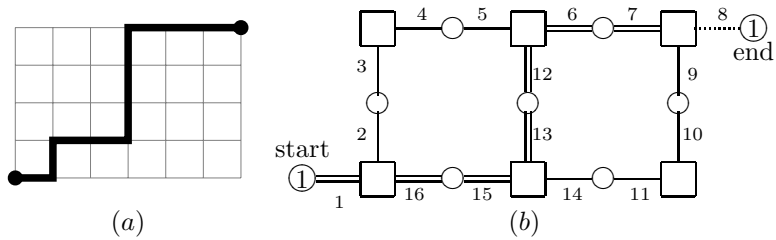
\begin{figure}[htb]
\[
\begin{array}{cc}
\begin{tikzpicture}
\draw[step=.5cm,gray,very thin] (0,0) grid (3,2);
  \draw[black, line width=.1cm] (0,0) circle (.05cm) -- (.5,0)  -- (.5,.5)  -- (1.5,.5) -- (1.5,2) -- (3,2) circle (.05cm);
\end{tikzpicture}
&
\begin{xy}<10mm,0mm>:
(0,2) *+{\phantom{I}\,}*\frm{-};
p+(1,0) *+{\;}*\frm{o} **@{-};
p+(1,0) *+{\phantom{I}\,}*\frm{-} **@{-};
  p+(0,-.83)  **@{=},
p+(1,0) *+{\;}*\frm{o} **@{=};
p+(1,0) *+{\phantom{I}\,}*\frm{-} **@{=};
  p+(1,0) *+{1}*\frm{o} **@{.}, 
  p+(1,-.3) *{\text{end}},
(0,1.75);
p+(0,-.75) *+{\;}*\frm{o} **@{-};
p+(0,-1) *+{\phantom{I}\,}*\frm{-} **@{-}; 
  p+(-1,0) *+{1}*\frm{o} **@{=}, 
  p+(-1,.4) *{\text{start}},
p+(1,0) *+{\;}*\frm{o} **@{=};
p+(1,0) *+{\phantom{I}\,}*\frm{-} **@{=};
p+(0,1) *+{\;}*\frm{o} **@{=},
p+(1,0) *+{\;}*\frm{o} **@{-};
p+(1,0) *+{\phantom{I}\,}*\frm{-} **@{-};
p+(0,1) *+{\;}*\frm{o} **@{-};
 p+(0,.75)  **@{-};
(-.5,-.3) *{{}_{1}};
p+(.3,.9) *{{}_{2}};
p+(0,.9) *{{}_{3}};
p+(.8,.7) *{{}_{4}};
p+(.8,0) *{{}_{5}};
p+(1.2,0) *{{}_{6}};
p+(.8,0) *{{}_{7}};
p+(1.2,0) *{{}_{8}};
p+(-.4,-.8) *{{}_{9}};
p+(0,-.8) *{{}_{10}};
p+(-.7,-.8) *{{}_{11}};
p+(-.9,0) *{{}_{14}};
{p+(-.4,.8) *{{}_{13}}; p+(0,.8) *{{}_{12}};},
p+(-1.2,0) *{{}_{15}};
p+(-.9,0) *{{}_{16}};
\end{xy}\\
(a)&(b)
\end{array}
\]
\caption{A monotone path (a) and the two squares in its lower left (b) showing the start point and underlying factor graph.  A satisfying path is shown, with edges carrying a one doubled.  Junctions are $AE^2$ and gates $\sPf\Xi$ for whichever $\Xi$ we are considering.} \label{fig:grids}
\end{figure}

We can evaluate weighted versions, force certain path segments to be included, introduce holes in the region, and so on at no computational cost by varying our six-parameter choice of $\Xi$ at each grid vertex.  Including $m$ start and end vertices counts families of nonintersecting monotone paths.  

Setting all the free  entries of the $\Xi$ corresponding to each internal vertex to one allows moves in all four directions, loops, and self-intersection.  Since $\Pf \Xi=2$ in this case, crossings contribute a multiplicative factor of $2$.  In other words a path contributes $2^{\#\text{self intersections}}$ to the sum over all paths.

We can check our computation for the $1 \times 2$ grid in Figure \ref{fig:grids}(b); fix the edge order shown in Figure \ref{fig:grids}(b).  
For both monotone and general paths, $\Theta$ has a one in the places $\{ (2,3), (4,5), (6,7), (9,10), (11,14),$ $ (12,13), (15,16) \}$.
To enforce the starting and ending conditions, we disallow $(2,16)$ and $(7,9)$ (otherwise we'll get the perimeter loop included) in $\Xi$. 
For general paths from start to finish, $\Xi$ has an $\xi_{i<j}=1, i,j \in I$ submatrix for subsets  $I \in \{(1,2), (1,16),(3,4), (5,6,12),$ $ (7,8,9), (10,11), (13,14,15)\}$.  There are no sign flips for $\Theta$.  Then $\Pf(\Theta+\Xi)=4$ as expected.  For monotone paths, $\Theta$ is the same but $\Xi$ is adjusted to exclude  moves $(13,14)$ and $(5,12)$ ,  $I \in \{(1,2),(1,16),(3,4), (5,6), (6,12),$ $ (7,8,9), (10,11), (13,15), (14,15)\}$.  Then $\Pf(\Theta+\Xi)=3$ as expected.  To count closed loops with no start or finish fixed, and each edge traversed at most once, we use $I \in \{(1,2,16),(3,4), (5,6,12), (7,8,9), (10,11), (13,14,15)\}$ for $\Xi$ and get  $\Pf(\Theta+\Xi)=4$: empty, left or right square, and perimeter.  Some Python code that performs the computation is as follows (here we take advantage of the fact no sign flips are needed).
{\small 
\begin{verbatim}
import numpy as np
from math import sqrt

onesabove = lambda n:np.triu(np.ones((n,n))) - np.tril(np.ones((n,n))) 
def count(xiones,thetaones,edges)
    A=np.zeros((edges,edges,int))
    for I in thetaones+xiones:
        J=[i-1 for i in I] #adjust to index from zero
        A[np.ix_(J,J)] = onesabove(len(J))
    return sqrt(np.linalg.det(A))

thetaones=[(2,3), (4,5), (6,7), (9,10), (11,14), (12,13), (15,16)]
#first the loops and crossings allowed version
xiones=[(1,2),(1,16),(3,4), (5,6,12),(7,8,9), (10,11), (13,14,15)]
#monotone version
xiones=[(1,2),(1,16),(3,4),(5,6),(6,12),(7,8,9),(10,11),(13,15),(14,15)]
#no start or finish version
xiones=[(1,2,16),(3,4), (5,6,12),(7,8,9), (10,11), (13,14,15)]
\end{verbatim}
}
\end{example}
Krattenthaler \cite{krattenthaler1997determinant}, reduces a certain count of lattice paths to evaluation of a Pfaffian in a particular case.  Such problems are related to Hilbert series of ladder determinental rings.  Our example shows how such algorithms (in fact formulae) for various lattice path counting problems can be derived in general using holographic algorithms.

\begin{example}
Bonin et al. \cite{bonin2003lattice} give a polynomial time algorithm for computing the Tutte polynomial of a lattice path matroid, refined and generalized to a $O(n^5)$ algorithm for multi-path matroids in \cite{bonin2007multi}.  Computing the Tutte polynomial of many classes of matroid, including arbitrary transversal or bicircular matroids, is $\mathsf{\#P}$-complete \cite{gimenez2006complexity}.   
We now show how a Pfaffian circuit can be applied to evaluate the Tutte polynomial of a lattice path matroid in $O(n^{\omega_p})$ time, a substantial improvement.  

Following \cite{bonin2003lattice}, a lattice path matroid is defined as follows.  Let $P$ and $Q$ be two lattice paths $(0,0)$ to $(m,r)$ which do not cross.  Let  $\{p_{u_1}, \dots p_{u_r}\}$, $u_1<u_2 \cdots u_r$ the set of North steps of $P$, and $\{p_{\ell_1}, \dots p_{\ell_r}\}$, $\ell_1 < \cdots \ell_r$ the set of North steps of $Q$.  Let $M[P,Q]$ be the transversal matroid with ground set $[m+r]=\{1, 2, \dots, m+r\}$ and presentation $(N_i: i \in [r])$ where $N_i$ are the sets of consecutive integers $(\ell_i, \dots u_i)$. 
Lattice path matroids have many nice properties, such as being closed under matroid duality and direct sums.  There are $\frac{1}{2}(C_n + {n \choose \ceil{n/2} }$ connected lattice path matroids on $n+1$ elements up to isomorphism.
A lattice path matroid is any matroid isomorphic to such a matroid. Because of the following theorem, we can just consider a lattice path matroid to be given by the paths inside the region defined by $P$ and $Q$.

\begin{theorem}[\cite{bonin2003lattice}]
A subset $B$ of $[m+r]$ with $|B|=r$ is a basis of $M[P,Q]$ if and only if the associated lattice path $P(B)$ stays in the region bounded by $P$ and $Q$.  Thus the number of bases of $M[P,Q]$ is the number of lattice paths from $(0,0)$ to $(m,r)$ that go neither below nor above $Q$.
\end{theorem}

The {\it Tutte polynomial} $t(M;x,y)$ of a matroid $M$ in terms of the internal $i(B)$ and external activities $e(B)$ of a basis is
\[
t(M; x,y) = \sum_{B \in \mathscr{B}(M)} x^{i(B)} y^{e(B)}.
\]
This is a useful formulation as we have the following.
\begin{theorem}[\cite{bonin2003lattice}]
Let $B$ be a basis of the lattice path matroid $M[P,Q]$ and let $P(B)$ be the associated path.  Then $i(B)$ is the number of times $P(B)$ meets the upper path $Q$ in a North step and $e(B)$ is the number of times $P(B)$ meets the lower path $P$ in an East step.
\end{theorem}

By selection of the $\Xi$ at each vertex, we have total freedom as to the weight of each possible two-edge path segment passing through any lattice point.  Thus we can just set the relevant $\xi_{ij}$ at the vertices along $P$ and $Q$ such that the value of the circuit is $t(M; x,y)$.  We conjecture that this approach can be extended to a fast algorithm for computing the Tutte polynomial of multipath matroids, where the path-basis relationship is more subtle.
\end{example}

\begin{example}[non-self-intersecting paths on a hexagonal grid]
Using an even trivalent Pfaffian gate (such as $\xi_{12}=\xi_{13}=\xi_{23}=1$), again paired with $AE^2$ junctions, leads to a hexagonal lattice (Figure \ref{fig:hexgrid}) where at each grid vertex only no path, or a path which both enters and exits, is allowed.  We also have not excluded separate closed loops. Thus we can count the (possibly weighted) number of non-self-intersecting paths (with extra loops) from a start vertex to a finish vertex in a connected region with $n$ hexagons in $O(n^{\omega_p})$ time.

\begin{figure}[htb]
\[
\begin{array}{cc}
\begin{tikzpicture}[scale=.50]
\grid{6}{8.7}{7};
\draw[x=.866 cm, line width=.1cm, black] 
(1,2.5) circle (.10cm) -- (2,2)  -- (3,2.5) -- (4,2)  -- (5,2.5) -- (5,3.5) 
-- (4,4)  -- (4,5) -- (5,5.5)  -- (6,5) -- (7,5.5) -- (7,6.5) 
-- (8,7) -- (9,6.5) -- (10,7) -- (10,8) -- (11,8.5) -- (11,9.5) circle (.10cm);
\end{tikzpicture}
&
\begin{xy}<7mm,0mm>:
(0,1) *+{\phantom{I}\,}*\frm{-}; 
 {p+(0,-1) *+{\;}*\frm{o} **@{-};   p+(0,-1) *+{\phantom{I}\,}*\frm{-} **@{-};p+(.86,-.86) *+{\;}*\frm{o} **@{=}; p+(.86,-.86) *+{\phantom{I}\,}*\frm{-}  **@{=};  p+(.86,.86) *+{\;}*\frm{o} **@{=};  p+(.86,.86) *+{\phantom{I}\,}*\frm{-}  **@{=};  p+(0,1) *+{\;}*\frm{o} **@{=}; p+(0,1) *+{\phantom{I}\,}*\frm{-}  **@{=};  p+(-.86,.86) *+{\;}*\frm{o} **@{=}; p+(-.56,.56) **@{=};  }, 
 {p+(.86,.86) *+{\;}*\frm{o} **@{-};  p+(.86,.86) *+{\phantom{I}\,}*\frm{-} **@{-};  p+(0,1) *+{\;}*\frm{o} **@{=}; p+(0,1) *+{\phantom{I}\,}*\frm{-}  **@{=};  p+(-.86,.86) *+{\;}*\frm{o} **@{-}; p+(-.86,.86) *+{\phantom{I}\,}*\frm{-}  **@{-}; p+(-.86,-.86) *+{\;}*\frm{o} **@{-};  p+(-.86,-.86) *+{\phantom{I}\,}*\frm{-} **@{-}; p+(0,-1) *+{\;}*\frm{o} **@{-};   p+(0,-.65)  **@{-};}, 
 {p+(-.86,.86) *+{\;}*\frm{o} **@{-};   p+(-.86,.86) *+{\phantom{I}\,}*\frm{-} **@{-};  p+(-.86,-.86) *+{\;}*\frm{o} **@{-}; p+(-.86,-.86) *+{\phantom{I}\,}*\frm{-}  **@{-}; p+(0,-1) *+{\;}*\frm{o} **@{-};   p+(0,-1) *+{\phantom{I}\,}*\frm{-} **@{-}; p+(-1,0) *+{1}*\frm{o} **@{=},  p+(.86,-.86) *+{\;}*\frm{o} **@{=}; p+(.86,-.86) *+{\phantom{I}\,}*\frm{-}  **@{=};  p+(.86,.86) *+{\;}*\frm{o} **@{=};  p+(.56,.56)  **@{=}; }, 
\end{xy}
\end{array}
\]

\caption{Non-self-intersecting path in a hexagonal grid and lower left portion of underlying factor graph. Gates on the boundary are (possibly scaled) $AE_2$, since there is only one path through them.} \label{fig:hexgrid}
\end{figure}
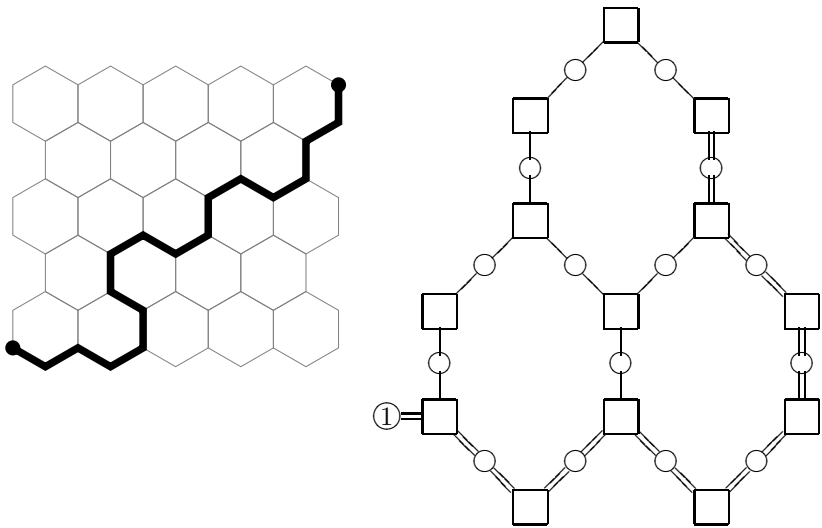
\end{example}
Because of the fast partition function available for such weighted path counting problems, estimating the $\Xi$ parameters of a statistical model over paths by maximum likelihood becomes feasible.

\section{Signatures Pfaffian after basis change} \label{sec:basischange}
We postpone the majority of this investigation to future work but mention a few results here.
\subsection{Symmetric signatures under homogeneous basis change}
The space of symmetric signatures under homogeneous basis change has been well studied in algebraic geometry for well over a century. Let $V$ be a two dimensional vector space; then the space of symmetric signatures up to scale is $\PP(\Sym^n V) \isom \PP^n$, and the action of a homogeneous change of basis is the action of $\PGL_2$.  The orbits have been computed in many cases.  The symmetric signatures which are realizable are precisely those which have an orbit representative in $S_+ \union S_-$. 

Symmetric signatures were analyzed in \cite{caisigpaper}.  We give a more concise geometric proof of some of these results. 
\begin{proposition}
All arity three symmetric signatures are Pfaffian after homogeneous basis change.
\end{proposition}
\begin{proof}
 In the arity three case \cite[Cor. 3.1]{caisigpaper}, which corresponds to $\PP(\Sym^3V)$, there are three $\PGL_2$ orbits.  The $\PGL_2$ action is the projective motions of twisted cubic $[v^3]$, the points $[v^2 \cdot w]$ with $v,w$ independent, and the points $[u\cdot v \cdot w]$, with $u,v,w$ pairwise independent.  

Thus to show that these are implementable under basis change relative to a set $\scrP$ of restricted predicates, (here the predicates which are Pfaffian in the standard basis $\scrP=S_+ \union S_-$), we just show that a representative of each orbit lies in $\scrP$.  This is clear since $|0\ket^{\circ 3}$ is Pfaffian, $|0\ket^{\circ 2}\circ |1\ket = |001\ket + |010\ket+|100\ket$ is odd Pfaffian, and $|0\ket \circ (|0\ket + |1\ket)  \circ (|0\ket - |1\ket)$ is Pfaffian.
\end{proof}

As $n$ increases the number of orbits is no longer finite.  In that case $PGL_2$ equivalence of points has a more complex description in terms of invariants (this is a subject of classical and geometric invariant theory).  That is, one can find invariants of the $PGL_2$ action such that two points lie in the same orbit iff the invariants take the same values (and that the ring of invariants is finitely generated was the motivation for the Hilbert basis theorem).  This means that a description of the attainable signatures can be had by considering the variety traced out by the invariants as we vary the point in $\scrP$.

\begin{example} \label{ex:Xmatchings}
The problem $\mathsf{\#}$Pl-Mon-Rtw-NAE-SAT was treated in \cite{HAWMG}.  We now give an example of a problem, $X$-matchings, from \cite{ValiantFOCS2004} that requires a homogeneous basis change, and couldn't be treated in \cite{HAWMG} because it requires the extensions to odd parity discussed above.
\begin{problem}[X-matchings \cite{ValiantFOCS2004}]
The input is a planar edge-weighted bipartite graph $(V,E,W)$ where the vertices $V=V_1 \union V_2$ are again partitioned in two, and the vertices in $V_1$ have degree 2.  The output is the sum of the weight of all matchings, where the weight of a matching is the product of the weights of the included edges, and $-(w_1 + \cdots + w_k)$ for all the unsaturated vertices in $V_2$, where $w_1, \dots, w_k$ are the weights of the edges incident to that unsaturated node.
\end{problem}

The change of basis for $X$-matchings, which is the same on all edges, is
\[
A=
\begin{pmatrix}
1&\phantom{-}1\\
0&-1\\
\end{pmatrix}
\qquad 
A^{-1}=
\begin{pmatrix}
1&1\\
0 &-1\\
\end{pmatrix}
\qquad 
A^{\vee}=
\begin{pmatrix}
1&0\\
1 &-1\\
\end{pmatrix}
\]
Let junctions (generators) represent $V1$ vertices in the original matching problem with $J=\bra 00 | +  \bra 01 | + \bra 10 |$.  The gates (recognizers) replace $V2$ vertices, where a degree $d$ vertex has gate $G_d=-d|0 \dots 0\ket + |0 \dots 01\ket + \dots + |1 0 \dots 0 \ket$.  Here we need to use a parity switch.  After basis change we obtain
\begin{eqnarray*}
A^{\vee \ot 2}(J)& =& (\bra0|+\bra 1|)^{\ot 2}   -(\bra0|+\bra1|)\bra1|  -\bra1|(\bra0|+\bra1|) \\
&=& \bra00| -\bra11|\\
A^{\ot 2}(G_2)& =& -2|00\ket + |0(0-1)\ket  + |(0-1)0\ket \\
&=&-2|00\ket +2|00\ket - |10\ket -|01\ket = |10\ket + |01\ket \\
A^{\ot d}(G_d)& =&  -d|0\ket^{\ot d} + d|0 \ket^{\ot d} - |0\dots 01\ket - \cdots - |10\dots 0\ket.\\
\end{eqnarray*}
Now 
\[
A^{\vee \ot 2}(J) = -\sPfd
\begin{pmatrix}
0&-1\\
1& 0\\
\end{pmatrix}, 
\qquad 
A^{\ot d}(G_d)=\;
\begin{xy}<10mm,0mm>:
(6,0)  *++{\Xi}*\frm{-};
  {p+(-.35,.25);p+(-1,0)  **@{-}},
  {p+(-.35,.20);p+(-1,0)  **@{-}},
  {p+(-.35,-.25);p+(-1,0)  **@{-}},
  p+(-.85,.08) *{\vdots},
  p+(1,0)  *+{1}*\frm{o}  **@{-},
  p+(-1.2,0) *{d};
\end{xy}
\]
where $A^{\ot d}(G_d)$ is the subPfaffian of the $(d+1)\times (d+1)$ matrix $\Xi$ with ones in the $(i,d+1)$th entries, $i<d+1$, and zeros for the rest of the free entries, contracted against the arity-one junction $\bra1|$.
\end{example}

\subsection{Pfaffian implementable predicates under free basis change}
We now consider the question of which predicates become compound Pfaffian after arbitrary, heterogeneous basis change over $\C$.  Since this corresponds to the $\times GL_2$ action in the binary case, it can be studied as a special case of the kind of orbit classification which follows from the work of Kronecker, and has been studied in the general Lie algebra context by Vinberg \cite{vinberg1975,vinberg_linear_1989}, Kac \cite{kac}, Kostant-Rallis, and others.  Parfenov \cite{parfenov_orbits_2001} spelled out the classification of spaces of tensors with finite orbits under a product of $GL$ actions in complete detail, with number of orbits, representatives, and the orbit abutment graph. This classification was rediscovered in the physics literature in the subsequent decade under the name ``classification of entangled states up to SLOCC equivalence.''

We have already seen that $\val(\Gamma)$ fills either the odd or even weight subspace, i.e. equals $\PP(S_+)$ or $\PP(S_-)$, for arity one or two. The consequence is that the entire orbit  of each parity subspace is accessible.  For arity three, there are six nonzero $\GL_2\times \GL_2 \times \GL_2$ orbits in $\C^2 \ot \C^2 \ot \C^2$.  It is well known that they have the following orbit representatives.
\begin{eqnarray*}
N_1&=&|000 \ket\\
N_2&=&|000\ket+|110\ket\\
N_3&=&|000\ket+|101\ket\\
N_4&=&|000\ket+|011\ket\\
N_5&=&|000\ket+|011\ket +|101\ket\\
N_6&=&|000\ket+|111\ket
\end{eqnarray*}
The first five are Pfaffian gates, and the last is Pfaffian in the Hadamard basis.  After possibly an $X$ change of basis, all are also realizable as junctions after some basis change.  Considering these gates or junctions as the generators in a $\mathsf{\#CSP}$ or circuit, we have the following.
\begin{theorem} \label{thm:freebasisarity3}
All arity three predicates are Pfaffian, possibly after a free basis change. 
\end{theorem}

\begin{corollary}
In particular, all SLOCC equivalence classes of tripartite entangled qubits can be generated as Pfaffian predicates under free basis change.
\end{corollary}

Of course, there are further restrictions required to obtain a Pfaffian circuit; these come into play when we want to compute with these states.  For arity four and above, there are already infinitely many orbits under arbitrary basis change and more sophisticated geometric techniques are needed to analyze the situation.  Thus we postpone this analysis to future work.

\section{Acknowledgments}
The author would like to thank J.-Y. Cai, D. Green, S. Hallgren, J.M. Landsberg, S. Norine, Y. Shi, A. Smith, B. Sturmfels, J. Turner, and L. Valiant for helpful discussions.

\pagebreak
\bibliography{Lmatrix}
\pagebreak

\end{document}